\documentclass[a4paper,11pt]{amsart}

\usepackage[latin1]{inputenc} 
\usepackage{graphicx}
\usepackage[pdftex,bookmarks]{hyperref}
\usepackage{amsmath}    
\usepackage{amssymb}
\usepackage{amsmath,amsthm}
\usepackage{textcomp}
\usepackage{color}
\usepackage[all]{xy}
\usepackage{geometry}
\usepackage{tikz-cd}

\newtheorem{teo}{Theorem}[section]
\newtheorem{defi}{Definition}[section]
\newtheorem{lemma}{Lemma}[section]

\newtheorem{cor}{Corollary}[section]
\newtheorem{prop}{Proposition}[section]

\newtheorem{rem} {Remark}[section]

\DeclareMathOperator{\im}{im}

\DeclareMathOperator{\dvol}{dvol}

\DeclareMathOperator{\vol}{vol}

\DeclareMathOperator{\End}{End}

\DeclareMathOperator{\tr}{tr}

\newcommand{\beq}{\begin{equation}}
	\newcommand{\eeq}{\end{equation}}
\newcommand{\beqn}{\begin{equation*}}
	\newcommand{\eeqn}{\end{equation*}}

\newcommand{\R}{\mathbb{R}}

\newcommand{\pp}{\partial\bar{\partial}}

\newcommand{\w}{\omega}

\newcommand{\mF}{\mathcal{F}}

\newcommand{\mE}{\mathcal{E}}
\newcommand{\mL}{\mathcal{L}}

\newcommand{\mQ}{\mathcal{Q}}
\newcommand{\mG}{\mathcal{G}}
\newcommand{\mH}{\mathcal{H}}

\DeclareMathOperator \rank{rank}

\DeclareMathOperator{\rd}{rd}

\DeclareMathOperator{\mov}{Mov}

\title[Simply connectedness of K\"ahler and Riemannian manifolds via spectral estimates]{Simply connectedness of K\"ahler and Riemannian manifolds via spectral estimates\\ (with an appendix by Shiyu Zhang)}
\author[F.Bei]{Francesco Bei} 
\address{Francesco Bei \\ Dipartimento di Matematica Guido Castelnuovo \\ Sapienza Universit\`a di Roma \\ Piazzale Aldo Moro 5 \\ I-00185 Roma.}
\email{bei@mat.uniroma1.it}
\address{Shiyu Zhang \\ School of Mathematical Sciences, University of Science and Technology of China, Hefei, 230026, P.R. China}
\email{{\tt shiyu123@mail.ustc.edu.cn}}

\begin{document}

\maketitle

\begin{abstract}
Let $(M,h)$ be a compact K\"ahler manifold. Under a rather weak spectral positivity assumption we prove that $M$ is rationally connected and thus simply connected, projective with $h^{p,0}(M)=\{0\}$ for each $p>0$. Then, in the second part of this paper, we focus on Riemannian manifolds and we provide an appropriate spectral positivity assumption which guarantees that a compact and oriented even dimensional Riemannian manifold $(M,g)$ is a simply connected real homology sphere. Finally, in the appendix, a characterization of the rational dimension of compact K\"ahler manifolds in terms of the positivity of the minimal slope of the tangent bundle is given.
\end{abstract}

\vspace{0.5 cm}



\section*{Introduction}

The interaction between curvature and topology is a deep and active research field that drew the attention of several mathematicians in the last decades. A landmark result in this research area is certainly the famous Bonnet-Myers theorem according to which a complete manifold $(M,g)$ with   positive Ricci curvature $\mathrm{Ric}_g$ is compact (with an upper estimates on the diameter) and has finite fundamental group. This theorem stimulated several further results where the uniform positivity of the Ricci curvature is replaced with weaker assumptions that allow the Ricci curvature to be somehow negative. Just to have a glimpse of the literature with no goal of completeness we can recall here \cite{AntoXu}, \cite{Aubry},\cite{CarRose}, \cite{S-Gallot}, \cite{Peter1}, \cite{Sprouse}, \cite{Wu}. In particular one way to relax the assumptions of the Bonnet-Myers theorem is to replace the uniform positivity of the Ricci curvature with  positivity in the spectral sense: more precisely it is required that the spectrum of a Schr\"odinger-type operator $\gamma\Delta+r$ with $\gamma$ a positive constant, $\Delta$ the Laplace-Beltrami operator of $(M,g)$ and $r:M\rightarrow \mathbb{R}$ the function that assigns to each $x\in M$ the lowest eigenvalue of $\mathrm{Ric}_g$ in $x$, is entirely positive. This is done for instance by Carron and Rose in \cite[Th. A]{CarRose} where the authors proved that $(M,g)$ is compact with an upper estimate on the diameter and with finite fundamental group provided  $\gamma\Delta+r$ has entirely positive spectrum, with $\gamma\in [0,1/(m-2)]$ and $m$ the dimension of $M$. This result has been recently improved by Antonelli and Xu in \cite[Th. 1]{AntoXu} where the authors obtained again that $(M,g)$ is compact with an upper estimate on the diameter and with finite fundamental group under the weaker and sharp assumption that $\gamma\Delta+r$ has entirely positive spectrum, with $\gamma\in [0,(m-1)/(m-2)]$. So far, we have only considered the interplay between  Ricci curvature and  topology in the realm of Riemannian manifolds. In the K\"ahler setting a beautiful result due to Kobayashi shows that the positivity of the Ricci curvature has even a stronger impact on the topology of the underlying manifold. More precisely in \cite[Th. A]{Koba} Kobayashi proved that a compact K\"ahler manifold with positive Ricci curvature is necessarily simply connected. Kobayashi's theorem inspired subsequent works where the simply connectedness of a compact K\"ahler manifold is attained through different notions of positive curvature see e.g. \cite[Th. 1.2]{Ni}, \cite[Th. 1.3]{Yang} and the references therein.\\
In this paper we aim to give an answer to the following question: given a compact K\"ahler manifold $(M,h)$, is it possible to achieve the same conclusion of Kobayashi's theorem by simply requiring positivity in the spectral sense of a suitable Schr\"odinger-type operator? It turns out that the above question has a positive answer under a very weak spectral positivity assumption. More precisely, by exploiting the notion and the geometric consequences of the rational connectedness, we have the following
\begin{teo}
\label{rationallyint}
Let $(M,J,h)$ be a complete K\"ahler manifold of real dimension $2m$.
\begin{enumerate}
\item If $m>1$ and the operator $$\gamma\Delta+r_1:L^2(M,h)\rightarrow L^2(M,h)$$ has entirely positive spectrum for some $\gamma\in [0,\frac{4}{2m-1})$ then $M$ is compact and rationally connected. 
\item If $M$ is compact and the operator $$\gamma\Delta+r_1:L^2(M,h)\rightarrow L^2(M,h)$$ has entirely positive spectrum for some $\gamma\in [0,+\infty)$ then $M$ is rationally connected.
\end{enumerate}
In particular $M$ is simply connected, projective and $h^{p,0}(M)=\{0\}$ for each $p=1,...,m$.
\end{teo}
 Roughly speaking $r_1$ is the function that assigns to each $x\in M$ the smallest eigenvalue in $x$ of the Ricci curvature of $(M,h)$. We refer to \eqref{function} for the precise definition of $r_1$. The crucial ingredient in proving the above result is to convert spectral positivity into the total positivity of \(r_1\), which implies a vanishing result and yields the rational connectedness. In the first section we also show that the non-negativity of $\gamma\Delta+r_1$ in the spectral sense with $\gamma\in [0,2)$ yields interesting geometric consequences. Namely:
\begin{teo}
\label{finitenessint}
Let $(M,J,h)$ be a compact K\"ahler manifold of real dimension $2m$ such that for some $\gamma\in [0,2)$ the operator $$\gamma\Delta+\frak{r}_p:L^2(M,h)\rightarrow L^2(M,h)$$ has entirely non-negative spectrum for each $p=1,...,m$. Then $M$ satisfies at least one of the following two properties:
\begin{enumerate}
\item $\pi_1(M)$ is finite,
\item  $\chi(M,\mathcal{O}_M)= 0$.
\end{enumerate}
\end{teo}
Also in this case we refer to \eqref{function} for the precise definition of $\frak{r}_p$. In the second and last section of this paper, we revolve our attention to the case of Riemannian manifolds. Motivated by the research of a Riemannian analogue of Th. \ref{rationallyint}, we looked for a spectral positivity assumption that implies the triviality of $\pi_1(M)$ and the vanishing of $H^k(M,\mathbb{R})$ on a compact Riemannian manifold $(M,g)$. This is partially accomplished in Th. \ref{simplyr}:
\begin{teo}
\label{simplyrint}
Let $(M,g)$ be a compact Riemannian manifold of dimension $m>1$ such  that the operator 
\begin{equation}
\label{wswsint}
\frac{m}{m-1}\Delta+w_{k,1}:L^2(M,g)\rightarrow L^2(M,g)
\end{equation}
 has entirely positive spectrum for each $k=1,...,m-1$. 
\begin{enumerate}
\item If $m$ is even and $M$ is orientable, then $M$ is a simply connected, real homology sphere.
\item  If $m$ is even and $M$ is non-orientable, then $M$ is a simply connected, real homology sphere with $\pi_1(M)\cong\mathbb{Z}_2$.
\item If $m$ is odd then $M$ is a real homology sphere. In particular $M$ is orientable.
\end{enumerate}
\end{teo}
In \eqref{wswsint} $w_{k,1}$ stands for the function that assigns to each $x\in M$ the smallest eigenvalue of the endomorphism $W_k:\Lambda^kT^*M\rightarrow \Lambda^kT^*M$ appearing in the Weitzenb\"ock formula with $k=1,...,m-1$. We refer to \eqref{eigenweit} for the precise definition. Besides the above result, in the second section we also prove a Riemannian analogue of Th. \ref{finitenessint}:
\begin{teo}
\label{finitenessrint}
Let $(M,g)$ be a compact Riemannian manifold of  dimension $m>1$. Assume that $\chi(M)\neq 0$ and the operator $$\gamma\Delta+w_{k,1}:L^2(M,g)\rightarrow L^2(M,g)$$ has entirely non-negative spectrum for some $\gamma\in [0,\frac{m}{m-1})$ and each $k=1,...,m-1$. Then $M$ has finite fundamental group.
\end{teo}

Unlike Th. \ref{rationallyint}, the proofs of Th. \ref{finitenessint}, Th. \ref{simplyrint} and Th. \ref{finitenessrint} are mainly based on direct $L^2$-vanishing theorems and the Atiyah $L^2$-index theorem. Then we continue by proving some sufficient conditions expressed in terms of the curvature operator that allow to apply Th. \ref{simplyrint} and Th. \ref{finitenessrint}. We refer to Th. \ref{curvop} and Th. \ref{curvop2} for  the precise statements. We then conclude the second section by collecting some sufficient conditions that guarantees the spectral positivity of \eqref{wswsint} without necessarily implying that $W_k$ is positive.\\
Finally,  the appendix contains a characterization of the rational dimension of a compact K\"ahler manifolds in terms of the positivity of the minimal slope of its tangent bundle.\\

\noindent \textbf{Acknowledgment.} The author wishes to thanks Simone Diverio and Shiyu Zhang for interesting conversations and emails. The author  acknowledge financial support from 2024 Sapienza research grant {\em New research trends in Mathematics at Castelnuovo} and  INdAM-GNSAGA Project, codice CUP E53C24001950001.

\section{The K\"ahler case}

Let $(M,J,h)$ be a K\"ahler manifold of real dimension $2m$. We denote with $\omega=h(\cdot,J\cdot)$ the corresponding K\"ahler form. Let $\mathrm{Ric}_h$ be the Ricci tensor of the metric $h$ and  let $\mathrm{ric}_h$ be the endomorphism of $TM$ such that $\mathrm{Ric}_h(\cdot,\cdot)=h(\mathrm{ric}_h \cdot, \cdot)$. Since $\mathrm{ric}_h$ commutes with $J$, the complex structure of $M$, each eigenspace of $\mathrm{ric}_h$ is preserved by $J$ and thus has even dimension. There exist therefore $m$ real-valued continuous functions $r_j:M\rightarrow \mathbb{R}$, $j=1,...,m$ such that 
$$
r_1\leq r_2\leq...\leq r_{m}$$
 and for each $x\in M$ 
\begin{equation}
\label{eigenricci}
\{r_1(x),r_1(x),r_2(x),r_2(x),...,r_{m}(x),r_{m}(x)\}
\end{equation}
is the set of eigenvalues of $\mathrm{ric}_{h,x}:T_xM\rightarrow T_xM$. Since $\mathrm{ric}_h$ commutes with $J$ his $\mathbb{C}$-linear extension to $TM\otimes \mathbb{C}$ preserves both $T^{1,0}M$ and $T^{0,1}M$. Let us denote with $\mathrm{ric}^{1,0}$ the restriction of $\mathrm{ric}_h$ to $T^{1,0}M$. Then for each $x\in M$ the $m$ eigenvalues of 
\begin{equation}
\label{eigenvalues-ricci}
\mathrm{ric}_x^{1,0}:T_x^{1,0}M\rightarrow T_x^{1,0}M
\end{equation}
 are given by $$\{r_1(x),r_2(x),...,r_{m-1}(x),r_{m}(x)\}.$$ Finally, with a little abuse of notation, we still denote with $\mathrm{ric}^{1,0}$ the endomorphism of $\Lambda^{1,0}(M)$ induced by $\mathrm{ric}^{1,0}:T^{1,0}M\rightarrow T^{1,0}M$. Let us denote with $\nabla:C^{\infty}(M,T^{1,0}(M))\rightarrow C^{\infty}(M,T^*M\otimes T^{1,0}(M))$ the Chern connection of $(M,J,h)$ and, with a little abuse of notation, let us still denote with $$\nabla:C^{\infty}(M,\Lambda^{p,0}(M))\rightarrow C^{\infty}(M,T^*M\otimes \Lambda^{p,0}(M))$$ the connection induced by the Chern connection. Furthermore let $$\Delta_{\overline{\partial},p,q}:\Omega^{p,q}(M)\rightarrow \Omega^{p,q}(M),\quad  \Delta_{\overline{\partial},p,q}:=\overline{\partial}_{p,q}^t\circ \overline{\partial}_{p,q}+\overline{\partial}_{p,q-1}\circ \overline{\partial}^t_{p,q-1}$$ be the Hodge-Kodaira Laplacian acting on $\Omega^{p,q}(M)$. Then the Weitzenb\"ock formula tells us that $$2\Delta_{\overline{\partial},p,0}=\nabla^t\circ \nabla+\mathrm{ric}_h^{p,0}$$ where $\mathrm{ric}^{p,0}:\Lambda^{p,0}(M)\rightarrow \Lambda^{p,0}(M)$ denotes the vector bundle endomorphism obtained by $\mathrm{ric}^{1,0}:\Lambda^{1,0}(M)\rightarrow \Lambda^{1,0}(M)$ extended as a derivation and $\nabla^t:C^{\infty}(M,T^*M\otimes \Lambda^{p,0}(M))\rightarrow C^{\infty}(M,\Lambda^{p,0}(M))$ is the formal adjoint of $\nabla:C^{\infty}(M,\Lambda^{p,0}(M))\rightarrow C^{\infty}(M,T^*M\otimes \Lambda^{p,0}(M))$ with respect to $h$, see e.g. \cite{Moroianu}. Since point by point the eigenvalues of $\mathrm{ric}^{1,0}:\Lambda^{1,0}(M)\rightarrow \Lambda^{1,0}(M)$ are given by $\{r_1(x),r_2(x),...,r_m(x)\}$ it is clear that the eigenvalues of $\mathrm{ric}^{p,0}_x$ are given by all the combinations $$r_{j_1}(x)+....+r_{j_p}(x)$$ with $j_1,...,j_p\in \{1,...,m\}$ and $j_1<j_2<...<j_p$. Let us now define the following continuous functions 
\begin{equation}
\label{function}
\frak{r}_p:M\rightarrow \mathbb{R},\quad\quad \frak{r}_p(x):=r_1(x)+...+r_p(x)
\end{equation}

 Note that $r_1=\frak{r}_1$, $pr_1\leq \frak{r}_p$ with $1\leq p\leq m$ and for any $p\in \{1,...,m\}$ and $\theta\in \Omega^{p,0}(M)$ we have $$h(\mathrm{ric}^{p,0}\theta,\theta)\geq h(\frak{r}_p\theta,\theta).$$
For various reasons that will become clear in the proof of the next result we are interested in the following Schr\"odinger-type differential operators: $\gamma\Delta+\frak{r}_p:C^{\infty}(M)\rightarrow C^{\infty}(M)$, with $\gamma$ a real, positive constant. Since $M$ is compact, $\frak{r}_p$ is continuous and real-valued and $\Delta$ is elliptic it is not difficult to verify that $\gamma\Delta+\frak{r}_p$ is essentially self-adjoint on $L^2(M,h)$. With a little abuse of notation we denote its $L^2$-closure as
\begin{equation}
\label{schr}
\gamma\Delta+\frak{r}_p:L^2(M,h)\rightarrow L^2(M,h)
\end{equation}
and we note, again by the reasons explained above, that the domain of \eqref{schr} equals the domain of the $L^2$-closure of the Laplace-Beltrami operator $\Delta$ and that the corresponding graph norms are equivalent. Moreover, again by the fact that $M$ is compact, $\frak{r}_p$ is real and continuous  and $\Delta$ is elliptic, we can easily deduce that \eqref{schr} has entirely discrete spectrum. We denote its eigenvalues with $$\sigma_1\leq \sigma_2\leq...\leq \sigma_k\leq...$$ We recall now that a compact, complex manifold $N$ is called {\em rationally connected} if any two points of $N$ can be joined by a chain of rational curves. We refer to the appendix or to \cite{De2001} for further details. We have all the ingredients for the first main result of this note. 

\begin{teo}
\label{rationally}
Let $(M,J,h)$ be a complete K\"ahler manifold of real dimension $2m$.
\begin{enumerate}
\item If $m>1$ and the operator $$\gamma\Delta+r_1:L^2(M,h)\rightarrow L^2(M,h)$$ has entirely positive spectrum for some $\gamma\in [0,\frac{4}{2m-1})$ then $M$ is compact and rationally connected. 
\item If $M$ is compact and the operator $$\gamma\Delta+r_1:L^2(M,h)\rightarrow L^2(M,h)$$ has entirely positive spectrum for some $\gamma\in [0,+\infty)$ then $M$ is rationally connected.
\end{enumerate}
\end{teo} 
To proving the above theorem, we need to show the following statement and a vanishing result.
\begin{lemma}
\label{totalp-spectralp}
Let $(M,g)$ be a compact Riemannian manifold, let $\Delta$ its Laplace-Beltrami operator and let $f:M\rightarrow \mathbb{R}$ be a continuous function. Then $$\int_Mf\dvol_g>0$$ if and only if there exists a constant $\gamma>0$ such that $$\gamma\Delta+f:L^2(M,g)\rightarrow L^2(M,g)$$ has entirely positive spectrum.
\end{lemma}
\begin{proof}
Clearly if $\gamma\Delta+f:L^2(M,g)\rightarrow L^2(M,g)$ has entirely positive spectrum then $\int_Mf\dvol_g>0$ as $$\int_Mf\dvol_g=\langle(\gamma\Delta+f)1,1\rangle_{L^2(M,g)}>0.$$ Conversely let us assume that $\int_Mf\dvol_g>0$. Without loss of generality we can assume that $f$ is smooth, as for every $\epsilon>0$ there exists $f_{\epsilon}\in C^{\infty}(M,\mathbb{R})$ such that $|f(x)-f_{\epsilon}(x)|<\epsilon.$ It is clear that for a sufficiently small $\delta>0$ we have $$P\leq \frac{\int_M \delta f \dvol_g}{\|\delta f\|_{L^{\infty}(M)}(4\vol_g(M)\|\delta f\|_{L^{\infty}(M)}+\int_M \delta f\dvol_g)}$$ with $P$ denoting the Poincar\'e constant of $(M,g)$, that is the inverse of the first non-zero eigenvalue of $\Delta$, see \cite[p. 626]{DabLock}. Then by \cite[Cor. 1.2]{DabLock} we can conclude that $\Delta+\delta f:L^2(M,g)\rightarrow L^2(M,g)$ has entirely positive spectrum and thus also $\frac{1}{\delta}\Delta+ f:L^2(M,g)\rightarrow L^2(M,g)$ has entirely positive spectrum.
\end{proof}
 
\begin{lemma}\label{prop-vanishing}
	Let $E$ be a holomorphic vector bundle of $\rank r$ on a compact complex manifold $X$ of real dimension $2m$. If $E$ is mean curvature   $(k-1)$-positive in the total sense, then for any pseudo-effective line bundle $L$ and $\forall k\leq p\leq m$, it holds that
	\begin{equation}\label{equa-vanishing}
		H^0(X,\Lambda^{p}E^{-1}\otimes L^{-1})=0.
	\end{equation}
\end{lemma}

We refer to the appendix for the precise definition of the total positivity of the mean curvature.

\begin{proof}
	For simplicity, we will always denote the K\"ahler form of a Hermitian metric $g$ by $\w_g$. Because $E$ is mean curvature $k$-positive in the total sense, then there exists a Hermitian metric $h$ on $E$ and a Gauduchon metric $g$ on $X$ such that 
	\begin{equation}\label{equa-condition-positivity}
		\int_X \frak{r}_{k,g}(E,H)\cdot\frac{\w_g^m}{m!}>0.
	\end{equation}	
	Let $s$ be a holomorphic section of $\Lambda^{p}E^{-1}\otimes L^{-1}$. Since $L$ is pseudo-effective, it admits a singular semi-positive metric $h_L$, i.e., $\im\pp\log\Theta_{h_L}\geq0$. One can choose a finite covering $\{U_\alpha\}$ consists of coordinates neighborhoods such that on each $U_\alpha$, $$s=\eta_\alpha\otimes e_\alpha^{-1}$$ for some $\eta_\alpha\in\Lambda^{p}E^{-1}|_{U_\alpha}$ and a generator $e_\alpha$ of $L|_{U_\alpha}$ with $|e_\alpha|_{h_L}^2=e^{-\varphi_\alpha}$ for some $\varphi_\alpha\in\mathrm{PSH}(U_\alpha)$. There is a induced singular Hermitian metric $H:=\Lambda^{p}H^{-1}\otimes h_L^{-1}$ on $\Lambda^{p}E^{-1}\otimes L^{-1}$ and
	\begin{align}
		|s|_H^2=|s|^2_{\Lambda^{p}H^{-1}\otimes h_L^{-1}}=e^{\varphi_\alpha}|\eta_\alpha|_{\Lambda^{p}H^{-1}}^2.
	\end{align}
We argue by the contradiction. Suppose that $s\neq 0$ and then $\eta_\alpha\neq 0$. Recall that we have the following Bochner-type formula in the sense of currents (see e.g. \cite[Proposition 3.4]{ZZ25}),
	\begin{equation}\label{equa-inequality}
		\begin{split}
			\im\pp\log|s|^2_H\wedge \frac{\w_g^{m-1}}{(m-1)!}&\geq \im\pp\varphi_\alpha\wedge \frac{\w_g^{m-1}}{(m-1)!}-\frac{\langle \tr_gR^{\Lambda^{p}H^{-1}}\eta_\alpha,\eta_\alpha\rangle_H}{|\eta_\alpha|^2_H}\cdot\frac{\w_g^m}{m!}\\
			&\geq\frak{r}_{p,g}(E,H)\cdot\frac{\w_g^m}{m!},
		\end{split}
	\end{equation}
	where the last inequality follows from that the smallest eigenvalue of $\tr_gR^{\Lambda^{p}H^{-1}}$ equals to $-\frak{r}_{k,g}(E,H)$. Integrate \eqref{equa-inequality} over $X$, one has $$\int_M \frak{r}_{p,g}(E,H)\cdot\frac{\w_g^m}{m!}\leq 0,$$
	which contradicts to \eqref{equa-condition-positivity}. The proof is completed.
	\end{proof}
As a direct consequence, we conclude

\begin{prop}\label{coro-rationaldimension}
	Let $X$ be a compact K\"ahler manifold of real dimension $2m$ such that the tangent bundle is total mean curvature   $(m-k)$-positive for some $1\leq k\leq m$. Then for every $k\leq p\leq m$ and any pseudo-effective line bundle $L$,
    $$H^0(X,\Omega^{p,0}X\otimes L^{-1})=0$$
    and thus $\rd(X)\geq k$. In particular, $X$ is rationally connected when $k=1$.
\end{prop}

\begin{proof}
	Vanishing results directly follows from Proposition \ref{prop-vanishing}. Let $f:X\dashrightarrow Y$ be a MRC fibration of $X$ and denote the indeterminacy locus of $f$ by $Z$. We argue by contradiction. If $\rd(X)\leq k-1$, then $p:=\dim Y\geq n-k+1$. By \cite{GHS}, any rational curves in $Y$ can be lifted into $X$ and thus $Y$ is not uniruled by the above (2). Then it follows from Ou's recent result \cite{ou25} that $K_Y$ is pseudo-effective. Applying Grauert-Remmert's extension theorem \cite{GR56}, we conclude that $f|_{X\setminus Z}^*K_Y$ extends to a pseudo-effective invertible subsheaf $\mF\subset \Lambda^{p,0}X$ since $\mathrm{codim}_XZ\geq 2$, which corresponds to a holomorphic section of $\Lambda^{p,0}X\otimes L^{-1}$ and we derive a contradiction.
\end{proof}

We can now prove Th. \ref{rationally}.
\begin{proof}
If $\gamma\Delta+r_1:L^2(M,h)\rightarrow L^2(M,h)$ has entirely positive spectrum for some $\gamma\in [0,\frac{4}{2m-1})$ then $M$ is compact, see  \cite[Cor. 1]{AntoXu}. Now, by Lemma \ref{totalp-spectralp}, we know that $\int_Mr_1\dvol_g>0$ and thus we can conclude that $M$ is rationally connected by Prop. \ref{coro-rationaldimension}. Analogously, if $\gamma\Delta+r_1:L^2(M,h)\rightarrow L^2(M,h)$ has entirely positive spectrum for some $\gamma\in [0, +\infty)$ then $\int_Mr_1\dvol_g>0$ and the conclusion follows again by Prop. \ref{coro-rationaldimension}. 
\end{proof}

\begin{cor}
\label{simply-proj}
Let $(M,J,h)$ be a K\"ahler manifold that satisfies the assumptions of Th. \ref{rationally}. Then $M$ is simply connected, projective and $h^{p,0}(M)=\{0\}$ for each $p=1,...,m$.
\end{cor}
\begin{proof}
It is well known that all the above properties follow from the fact that $M$ is rationally connected. More precisely we know that $H^{p,0}_{\overline{\partial}}(M)=\{0\}$ for each $p>0$, see e.g. Prop. \ref{coro-rationaldimension}. Since $M$ is compact and K\"ahler with $h^{p,0}(M)=0$ we know that $M$ is projective, see \cite[Th. 8.3]{Kodaira}. Finally, given that $M$ is projective and rationally connected we know that it is simply connected, see \cite[Cor. 4.29]{De2001}.
\end{proof}

We have now an immediate consequence of Th. \ref{rationally}:

\begin{cor} In the setting of Th. \ref{rationally}, let $f:M\rightarrow M$ be an arbitrarily fixed holomorphic map. Then $$\{x\in M : f(x)=x\}\neq \emptyset.$$
\end{cor}

\begin{proof}
Let us consider the Dolbeault complex $$0\rightarrow C^{\infty}(M)\stackrel{\overline{\partial}_{0,0}}{\rightarrow}\Omega^{0,1}(M)\stackrel{\overline{\partial}_{0,1}}{\rightarrow}...\stackrel{\overline{\partial}_{0,m-1}}{\rightarrow}\Omega^{0,m}(M)\rightarrow 0$$
Then, by Cor. \ref{simply-proj}, we have $H^{0,q}_{\overline{\partial}}(M)=\{0\}$ and consequently $$L_{\overline{\partial}}(f):=\sum_{q=0}^m(-1)^q\mathrm{Tr}\left(f^*: H^{0,q}_{\overline{\partial}}(M)\rightarrow H^{0,q}_{\overline{\partial}}(M)\right)=1.$$
The conclusion now follows from the holomorphic Lefschetz fixed point theorem, see e.g. \cite[Prop. 10.8]{Roe}.
\end{proof}

If we replace the positivity of the spectrum with the non-negativity and we compensate this by requiring a smaller constant in front of the Laplace-Beltrami operator we still have interesting geometric applications. More precisely:

\begin{teo}
\label{finiteness}
Let $(M,J,h)$ be a compact K\"ahler manifold of real dimension $2m$ such that for some $\gamma\in(0,2)$ the operator $$\gamma\Delta+\frak{r}_p:L^2(M,h)\rightarrow L^2(M,h)$$ has entirely non-negative spectrum for each $p=1,...,m$. Then $M$ satisfies at least one of the following two properties:
\begin{enumerate}
\item $\pi_1(M)$ is finite,
\item  $\chi(M,\mathcal{O}_M)= 0$.
\end{enumerate}
\end{teo}

\begin{proof}
In order to prove  this theorem it is enough to prove that if $\pi_1(M)$ is infinite then $\chi(M,\mathcal{O}_M)=0$. So, let us assume that $\pi_1(M)$ is infinite. Let $\pi:\tilde{M}\rightarrow M$ be the universal covering of $M$ and let us consider the operator $$\gamma\tilde{\Delta}+\tilde{\frak{r}}_p:C^{\infty}_c(\tilde{M})\rightarrow C_c^{\infty}(\tilde{M})$$ with $\tilde{\frak{r}}_p:=\frak{r}_p\circ\pi$ and $p=1,...,m$. Since $\tilde{\Delta}:C^{\infty}_c(\tilde{M})\rightarrow C_c^{\infty}(\tilde{M})$  is essentially self-adjoint when viewed as an unbounded and densely defined operator acting on $L^2(\tilde{M},\tilde{h})$ and given that $\tilde{\frak{r}}_p$ is a bounded real-valued function on $\tilde{M}$, we can conclude that $\gamma\tilde{\Delta}+\tilde{\frak{r}}_p:C^{\infty}_c(\tilde{M})\rightarrow C_c^{\infty}(\tilde{M})$ is also essentially self-adjoint in $L^2(\tilde{M},\tilde{h})$ for each $1\leq p\leq m$. With a little abuse of notation let us denote with 
\begin{equation}
\label{lift}
\gamma\tilde{\Delta}+\tilde{\frak{r}}_p:L^2(\tilde{M},\tilde{h})\rightarrow L^2(\tilde{M},\tilde{h})
\end{equation}
 the $L^2$-closure of  $\gamma\tilde{\Delta}+\tilde{\frak{r}}_p:C^{\infty}_c(\tilde{M})\rightarrow C_c^{\infty}(\tilde{M})$. Since $\tilde{\frak{r}}_p$ is bounded it is also easy to deduce that the domain of \eqref{lift} equals the domain of the $L^2$-closure of the Laplace-Beltrami operator $\tilde{\Delta}$ and that the corresponding graph norms are equivalent. Note that using the pushing-down technique, see \cite[\S 4]{BMP}, or alternatively using \cite[Th. 1]{FCS} or \cite[Lemma 3.10]{PRS},  we can conclude that \eqref{lift} has still non-negative spectrum.  Let us now denote with $\tilde{\Delta}_{\overline{\partial},p,q}:\Omega^{p,q}(\tilde{M})\rightarrow \Omega^{p,q}(\tilde{M})$ the Hodge-Kodaira Laplacian on $\tilde{M}$ and let 
\begin{equation}
\label{L2closure}
\tilde{\Delta}_{\overline{\partial},p,q}:L^2\Omega^{p,q}(\tilde{M},\tilde{h})\rightarrow L^2\Omega^{p,q}(\tilde{M},\tilde{h})
\end{equation}
 be the $L^2$-closure of $\tilde{\Delta}_{\overline{\partial},p,q}:\Omega_c^{p,q}(\tilde{M})\rightarrow L^2\Omega^{p,q}(\tilde{M},\tilde{h})$. We denote with $\mathcal{H}^{p,q}_{2,\overline{\partial}}(\tilde{M},\tilde{h})$ the kernel of \eqref{L2closure}. Let now $\psi\in \mathcal{H}^{p,0}_{2,\overline{\partial}}(\tilde{M},\tilde{h})$ and let $W^{1,2}(\tilde{M},\Lambda^{p,0}(\tilde{M}))$ be the Sobolev space of $L^2$ $(p,0)$-forms $\eta$ such that $\nabla \eta\in L^2(\tilde{M},T^*\tilde{M}\otimes  \Lambda^{p,0}(\tilde{M}), \tilde{h})$. Since $M$ is compact and $(\tilde{M},\tilde{h})$  is the universal covering of $M$ endowed with the pull-back metric, we know that $\psi\in W^{1,2}(\tilde{M},\Lambda^{p,0}(\tilde{M}))$, see \cite[Th.2]{Dod}. Consequently $|\psi|_h\in W^{1,2}(\tilde{M},\tilde{h})$, see e.g. \cite[Prop. 3.1]{BeiSob} where $W^{1,2}(\tilde{M},\tilde{h})$ denotes the Sobolev space of $L^2$ functions on $(\tilde{M},\tilde{h})$ with $L^2$ gradient. Hence, using the refined Kato inequality \cite[Cor. 4.3]{Cibo}, we obtain
$$
\begin{aligned}
0&=\langle 2\tilde{\Delta}_{\overline{\partial},p,0}\psi,\psi\rangle_{L^2\Omega^{p,0}(\tilde{M},\tilde{h})}\\
&=\int_{\tilde{M}}\left(|\tilde{\nabla} \psi|^2_{\tilde{h}}+\tilde{h}(\tilde{\mathrm{ric}}^{p,0}\psi,\psi)\right)\dvol_{\tilde{h}}\\
&\geq \int_{\tilde{M}}\left(2|\tilde{d}_0|\psi|_{\tilde{h}}|^2_{\tilde{h}}+\tilde{\frak{r}}_p|\psi|^2_{\tilde{h}}\right)\dvol_{\tilde{h}}\\
&\geq \int_{\tilde{M}}\left((2-\gamma)|\tilde{d}_0|\psi|_{\tilde{h}}|^2_{\tilde{h}}+\gamma|\tilde{d}_0|\psi|_{\tilde{h}}|^2_{\tilde{h}}+\tilde{\frak{r}}_p|\psi|^2_{\tilde{h}}\right)\dvol_{\tilde{h}}\\
&\geq \int_{\tilde{M}}(2-\gamma)|\tilde{d}_0|\psi|_{\tilde{h}}|^2_{\tilde{h}}\dvol_{\tilde{h}}\\
&\geq 0.
\end{aligned}
$$ We can thus conclude that $|\psi|_{\tilde{h}}$ is a constant function on $\tilde{M}$ and since  $|\psi|_{\tilde{h}}\in L^2(\tilde{M},\tilde{h})$ and $(\tilde{M},\tilde{h})$ has infinite volume, we can deduce that $\psi\equiv 0$. This shows that $\mathcal{H}^{p,0}_{2,\overline{\partial}}(\tilde{M},\tilde{h})=\{0\}$ for each $p=1,...,m$. Moreover, we also know that  $\mathcal{H}^{0,0}_{2,\overline{\partial}}(\tilde{M},\tilde{h})=\{0\}$, since $\mathcal{H}^{0,0}_{2,\overline{\partial}}(\tilde{M},\tilde{h})$ coincides with the space of $L^2$-harmonic functions on $(\tilde{M},\tilde{h})$, which is trivial given that $(\tilde{M},\tilde{h})$ is complete and has infinite volume. 
Now, using the conjugation and the Hodge star operator, we obtain that 
\begin{equation}
\label{cucu}
\mathcal{H}^{0,q}_{2,\overline{\partial}}(\tilde{M},\tilde{h})=\{0\}
\end{equation} for each $q=0,...,m$. According to the celebrated Atiyah's $L^2$-index theorem, see  \cite[Th. 3.8]{Atiyah}, we know that 
\begin{equation}
\label{L2-A}
\chi(M,\mathcal{O}_M)=\sum_{k=0}^m(-1)^k\dim_{\pi_1(M)}\left(\mathcal{H}^{0,q}_{2,\overline{\partial}}(\tilde{M},\tilde{h})\right).
\end{equation}
The above expression, that is $\dim_{\pi_1(M)}\left(\mathcal{H}^{0,q}_{2,\overline{\partial}}(\tilde{M},\tilde{h})\right)$, denotes the Von Neumann dimension of $\mathcal{H}^{0,q}_{2,\overline{\partial}}(\tilde{M},\tilde{h})$ with respect to $\pi_1(M)$. For the purpose of this proof it is not necessary to recall the definition of $\dim_{\pi_1(M)}\left(\mathcal{H}^{0,q}_{2,\overline{\partial}}(\tilde{M},\tilde{h})\right)$ for which we refer to \cite{Atiyah}. It is only enough to keep in mind that  $\dim_{\pi_1(M)}\left(\mathcal{H}^{0,q}_{2,\overline{\partial}}(\tilde{M},\tilde{h})\right)$ is a non-negative real number such that 
\begin{equation}
\label{L2v}
\dim_{\pi_1(M)}\left(\mathcal{H}^{0,q}_{2,\overline{\partial}}(\tilde{M},\tilde{h})\right)=0\quad \mathrm{if\ and\ only\ if}\quad  \mathcal{H}^{0,q}_{2,\overline{\partial}}(\tilde{M},\tilde{h})=\{0\}.
\end{equation}
Finally, by \eqref{cucu}, \eqref{L2-A} and \eqref{L2v}, we can conclude that $\chi(M,\mathcal{O}_M)=0$, as required.
\end{proof}

We have now the following application:

\begin{cor}
In the setting of Th. \ref{finiteness} every holomorphic $p$-form on $M$ has constant length and consequently $$h^{p,0}(M)\leq \frac{m!}{p!(m-p)!}$$ for each $p=1,...,m$.
\end{cor}
\begin{proof}
Let $\psi$ be a holomorphic $p$-form, with $1\leq p\leq m$. Arguing as in the proof of Th. \ref{finiteness} we know that $\psi\in W^{1,2}(M,\Lambda^{p,0}(M))$ and consequently $|\psi|_h\in W^{1,2}(M,h)$. Applying the refined Kato inequality \cite[Cor. 4.3]{Cibo}, we obtain
$$
\begin{aligned}
0&=\langle 2\Delta_{\overline{\partial},p,0}\psi,\psi\rangle_{L^2\Omega^{p,0}(M,h)}\\
&=\int_{M}\left(|\nabla \psi|^2_{h}+h(\mathrm{ric}^{p,0}\psi,\psi)\right)\dvol_{h}\\
&\geq \int_M\left(2|d_0|\psi|_{h}|^2_{h}+\frak{r}_p|\psi|^2_{h}\right)\dvol_{h}\\
&\geq \int_{M}\left((2-\gamma)|d_0|\psi|_{h}|^2_{h}+\gamma|d_0|\psi|_{h}|^2_{h}+\frak{r}_p|\psi|^2_{h}\right)\dvol_{h}\\
&\geq \int_{M}(2-\gamma)|d_0|\psi|_{h}|^2_{h}\dvol_{h}\\
&\geq 0.
\end{aligned}
$$ We can thus conclude that $|\psi|_{h}$ is a constant function on $M$. Finally let $\psi_1,...,\psi_{\ell}$ be $\ell$-holomorphic and linearly independent $p$-forms. If $\ell>\frac{m!}{p!(m-p)!}$ then there exists a point $x\in M$ and $\mu_1,...,\mu_{\ell}\in \mathbb{C}$ such that $\mu_1\psi_1(x)+...+\mu_{\ell}\psi_{\ell}(x)=0$ and $\mu_j\neq 0$ for some $j=1,...,\ell$.  Let $\psi:=\mu_1\psi_1+...+\mu_{\ell}\psi_{\ell}$. Since $\psi$ is holomorphic we know that $|\psi|_h$ is constant. Thus $\psi(x)=0$ implies that $\psi\equiv 0$. This contradicts the fact that $\psi_1,...,\psi_{\ell}$ are linearly independent. We can thus conclude that $\ell\leq \frac{m!}{p!(m-p)!}$, as required.
\end{proof}

Sufficient conditions to apply Th. \ref{finiteness} can be given entirely in terms of $r_1$, the lowest eigenvalue of the Ricci curvature \eqref{eigenricci}. Namely:

\begin{cor}
Let $(M,J,h)$ be a compact K\"ahler manifold of real dimension $2m$. If  the operator $$\gamma\Delta+mr_1:L^2(M,h)\rightarrow L^2(M,h)$$ has entirely non-negative spectrum for some $\gamma\in [0,2)$ and $\chi(\mathcal{O}_M)\neq 0$ then $\pi_1(M)$ is finite.

\end{cor} 
\begin{proof}
Note that $\gamma\Delta+mr_1:L^2(M,h)\rightarrow L^2(M,h)$ has entirely non-negative spectrum if and only if $\frac{\gamma}{m}\Delta+r_1:L^2(M,h)\rightarrow L^2(M,h)$ has entirely non-negative spectrum. Since $\frac{\gamma}{p}\geq \frac{\gamma}{m}$ we can deduce that if $\frac{\gamma}{m}\Delta+r_1:L^2(M,h)\rightarrow L^2(M,h)$ has entirely non-negative spectrum then also the spectrum of $\frac{\gamma}{p}\Delta+r_1:L^2(M,h)\rightarrow L^2(M,h)$ is entirely non-negative. This latter property amounts to saying that $\gamma\Delta+pr_1:L^2(M,h)\rightarrow L^2(M,h)$ has entirely non-negative spectrum and finally we can conclude that also the spectrum of $\gamma\Delta+\frak{r}_p:L^2(M,h)\rightarrow L^2(M,h)$ is entirely non-negative given that $\frak{r}_p\geq pr_1$. The conclusion follows now from Th. \ref{finiteness}.
\end{proof}

\section{The Riemannian case}
Let us now consider a compact Riemannian manifold $(M,g)$ of dimension $m$. We denote with $\nabla:C^{\infty}(M,TM)\rightarrow C^{\infty}(M,T^*M\otimes TM)$ the Levi-Civita connection and with a little abuse of notation let us denote with $\nabla:C^{\infty}(M,\Lambda^kT^*M)\rightarrow C^{\infty}(M,T^*M\otimes \Lambda^kT^*M)$ the connection induced by the Levi-Civita connection. Now let $\Delta_k:\Omega^k(M)\rightarrow \Omega^k(M)$ be the Hodge Laplacian. Then, the Weitzenb\"ock formula, see e.g. \cite{Gallot}, \cite{Poor}, tells us that $$\Delta_k=\nabla^t\circ \nabla+W_k$$ with $\nabla^t$ the formal adjoint of $\nabla:C^{\infty}(M,\Lambda^kT^*M)\rightarrow C^{\infty}(M,T^*M\otimes \Lambda^kT^*M)$ and $W_k$ an endomorphism of $\Lambda^kT^*M$ obtained from the curvature operator of $(M,g)$.
Then, for each fixed $k\in \{1,...,m-1\}$, there exist $\ell(k)$ real-valued continuous functions $w_{k,j}:M\rightarrow \mathbb{R}$, $j=1,...,\ell(k)$, with $\ell(k)=m!/(k!(m-k)!)$, such that 
$$
w_{k,1}\leq w_{k,2}\leq...\leq w_{k,\ell(k)}$$
 and for each $x\in M$ 
\begin{equation}
\label{eigenweit}
\{w_{k,1}(x),w_{k,2}(x),...,w_{k,\ell(k)}(x)\}
\end{equation}
is the set of eigenvalues of $W_{k,x}:\Lambda^kT^*_xM\rightarrow \Lambda^kT^*_xM$.   
Clearly for any $k\in \{1,...,m-1\}$ and $\eta\in \Omega^k(M)$ we have \begin{equation}
\label{triple}
g(W_k\eta,\eta)\geq g(w_{k,1}\eta,\eta).
\end{equation}
In this section we are interested in the following Schr\"odinger-type differential operators: $$\frac{m}{m-1}\Delta+w_{k,1}:C^{\infty}(M)\rightarrow C^{\infty}(M)$$ with $k=1,...,m-1$. Also in this case, given that $M$ is compact, $w_{k,1}$ is real and continuous and $\Delta$ is elliptic, we know that $\frac{m}{m-1}\Delta+w_{k,1}:C^{\infty}(M)\rightarrow L^2(M,g)$ carries a unique $L^2$-closed extension whose domain equals the domain of $\Delta:L^2(M,g)\rightarrow L^2(M,g)$, the $L^2$-closure of $\Delta:C^{\infty}(M)\rightarrow L^2(M,g)$. With little abuse of notation we denote it with
\begin{equation}
\label{schrw}
\frac{m}{m-1}\Delta+w_{k,1}:L^2(M,g)\rightarrow L^2(M,g).
\end{equation}
Moreover, again by the fact that $M$ is compact, $w_{k,1}$ is real and continuous and $\Delta$ is elliptic, we know that \eqref{schrw} has entirely discrete spectrum. We denote its eigenvalues with $$\lambda_{k,1}\leq \lambda_{k,2}\leq...\leq \lambda_{k,j}\leq...$$
We have now the first main result of this section: 
\begin{teo}
\label{simplyr}
Let $(M,g)$ be a compact Riemannian manifold of dimension $m>1$ such  that the operator 
\begin{equation}
\label{wsws}
\frac{m}{m-1}\Delta+w_{k,1}:L^2(M,g)\rightarrow L^2(M,g)
\end{equation}
 has entirely positive spectrum for each $k=1,...,m-1$. 
\begin{enumerate}
\item If $m$ is even and $M$ is orientable, then $M$ is a simply connected, real homology sphere.
\item  If $m$ is even and $M$ is non-orientable, then $M$ is a simply connected, real homology sphere with $\pi_1(M)\cong\mathbb{Z}_2$.
\item If $m$ is odd then $M$ is a real homology sphere. In particular $M$ is orientable.
\end{enumerate}
\end{teo}

\begin{proof}
\noindent \textbf{First part.} We divided it in four steps.\\
\textbf{First step}. $M$ is a real homology sphere.\\
Since we assumed that \eqref{wsws} has entirely positive spectrum we know that $\lambda_{k,1}$, the smallest eigenvalue of \eqref{wsws}, is positive for each $k=1,...,m-1$. Furthermore, we have already observed that the domain of \eqref{wsws} equals the domain of the $L^2$-closure of the Laplace-Beltrami operator $\Delta$ and that the corresponding graph norms are equivalent. Therefore, from the min-max characterization of the eigenvalues we can deduce that $$\frac{m}{m-1}\|d_0\phi\|^2_{L^2\Omega^1(M,g)}+\langle w_{k,1}\phi,\phi\rangle_{L^2(M,g)}\geq \lambda_{k,1}\|\phi\|^2_{L^2(M,g)}$$ for each $\phi\in W^{1,2}(M,g)$, with $W^{1,2}(M,g)$ denoting the Sobolev space of $L^2$ functions on $(M,g)$ with $L^2$ gradient. Let us now consider an arbitrarily fixed harmonic $k$-form $\eta$ for some $k\in \{1,...,m-1\}$. We know that $|\eta|_g\in W^{1,2}(M,g)$, see e.g. \cite[Prop. 3.1]{BeiSob}. In this way, keeping in mind the refined Kato inequality \cite[Th. 6.3]{CAGH} or \cite[Th. 3.8]{Cibo} and the min-max characterization of the eigenvalues, we have:
$$
\begin{aligned}
0&=\langle \Delta_k\eta,\eta\rangle_{L^2\Omega^k(M,g)}\\
&=\int_{M}\left(|\nabla \eta|^2_g+g(W_k\eta,\eta)\right)\dvol_{g}\\
&\geq \int_{M}\left(\frac{m}{m-1}|d_0|\eta|_g|^2_g+w_{k,1}|\eta|^2_g\right)\dvol_{g}\\
&\geq \lambda_{k,1}\|\eta\|^2_{L^2(M,g)}.
\end{aligned}
$$
Since $\lambda_{k,1}>0$ we can conclude that $\eta\equiv0$ and thus $H^k(M,\mathbb{R})=\{0\}$ for each $k=1,...,m-1$. Moreover $H^m(M,\mathbb{R})\cong H^1(M,\mathbb{R})\cong \mathbb{R}$ as $M$ is compact and oriented. This shows that $M$ is a real homology sphere and, using the fact that  it is even dimensional, we can also conclude that $\chi(M)=2$.\\

\noindent \textbf{Second step:} $\mathcal{H}^{k}_{2}(\tilde{M},\tilde{g})=\{0\}$ for each $k=1,...,m-1$.\\
 Let $\pi:\tilde{M}\rightarrow M$ be the universal covering of $M$. Let us consider the operator 
\begin{equation}
\label{fvfv}
\frac{m}{m-1}\tilde{\Delta}+\tilde{w}_{k,1}:C^{\infty}_c(\tilde{M})\rightarrow C_c^{\infty}(\tilde{M})
\end{equation}
with $\tilde{w}_{k,1}:=w_{k,1}\circ \pi$, $\tilde{\Delta}:C^{\infty}_c(\tilde{M})\rightarrow C_c^{\infty}(\tilde{M})$ the Laplace-Beltrami operator of $(\tilde{M},\tilde{g})$ and $\tilde{g}=\pi^*g$. Arguing as in the proof of Th. \ref{finiteness} we know that \eqref{fvfv} is  essentially self-adjoint in $L^2(\tilde{M},\tilde{g})$. Let us denote with 
\begin{equation}
\label{liftz}
\frac{m}{m-1}\tilde{\Delta}+\tilde{w}_{k,1}:L^2(\tilde{M},\tilde{g})\rightarrow L^2(\tilde{M},\tilde{g})
\end{equation}
its $L^2$-closure. Since $\tilde{w}_{k,1}$ is bounded, it is  easy to deduce that the domain of \eqref{liftz} equals the domain of the $L^2$-closure of the Laplace-Beltrami operator $\tilde{\Delta}$ and that the corresponding graph norms are equivalent. Furthermore using again the pushing-down technique, see \cite[\S 4]{BMP}, or alternatively using \cite[Th. 1]{FCS} or \cite[Lemma 3.10]{PRS},  we can conclude that \eqref{liftz} has entirely positive spectrum. In particular there exists $\tilde{\lambda}_{k,1}\geq \lambda_{k,1}>0$ such that $$\frac{m}{m-1}\|\tilde{d}_0\phi\|^2_{L^2\Omega^1(\tilde{M},\tilde{g})}+\langle\tilde{w}_{k,1}\phi,\phi\rangle_{L^2(\tilde{M},\tilde{g})}\geq \tilde{\lambda}_{k,1}\|\phi\|^2_{L^2(\tilde{M},\tilde{g})}$$ for each $\phi\in W^{1,2}(\tilde{M},\tilde{g})$. Let us now denote with $\tilde{\Delta}_{k}:\Omega^{k}(\tilde{M})\rightarrow \Omega^{k}(\tilde{M})$ the Hodge Laplacian on $\tilde{M}$ and let 
\begin{equation}
\label{L2closureH}
\tilde{\Delta}_{k}:L^2\Omega^k(\tilde{M},\tilde{g})\rightarrow L^2\Omega^{k}(\tilde{M},\tilde{g})
\end{equation}
 be the $L^2$-closure of $\tilde{\Delta}_{k}:\Omega_c^{k}(\tilde{M})\rightarrow L^2\Omega^{k}(\tilde{M},\tilde{g})$. We label with $\mathcal{H}^{k}_{2}(\tilde{M},\tilde{g})$ the kernel of \eqref{L2closureH}. Let now $\psi\in \mathcal{H}^{k}_{2}(\tilde{M},\tilde{g})$ with $k\in \{1,...,m-1\}$. By arguing again as in the  proof of Th. \ref{finiteness} we know that  $\psi\in W^{1,2}(\tilde{M},\Lambda^k(\tilde{M}))$ and $|\psi|_{\tilde{g}}\in W^{1,2}(\tilde{M},\tilde{g})$. In this way, using again the Weitzenb\"ock formula and the refined Kato inequality, we  obtain
$$
\begin{aligned}
0&=\langle \tilde{\Delta}_k\psi,\psi\rangle_{L^2\Omega^{k}(\tilde{M},\tilde{g})}\\
&\geq \int_{M}\left(\frac{m}{m-1}|\tilde{d}_0|\psi|_{\tilde{g}}|^2_{\tilde{g}}+\tilde{w}_{k,1}|\psi|^2_{\tilde{g}}\right)\dvol_{\tilde{g}}\\
&\geq \tilde{\lambda}_{k,1}\|\psi\|^2_{L^2(\tilde{M},\tilde{g})}.
\end{aligned}
$$ We can thus conclude that $\mathcal{H}^{k}_{2}(\tilde{M},\tilde{g})=\{0\}$ for each $k=1,...,m-1$.\\

\noindent \textbf{Third step:} $\pi_1(M)$ is finite.\\
 We prove this step by contradiction. Let us assume that $\pi_1(M)$ is infinite and, as above, let $\pi:\tilde{M}\rightarrow M$ be the universal covering of $M$. We proved in the second step that $\mathcal{H}^{k}_{2}(\tilde{M},\tilde{g})=\{0\}$ for each $k=1,...,m-1$. Since $\pi_1(M)$ is infinite we also know that $(\tilde{M},\tilde{g})$ is a complete Riemannian manifold of infinite volume. Therefore   $\mathcal{H}^{0}_{2}(\tilde{M},\tilde{g})=\{0\}$ and since $\tilde{M}$ is oriented we can use the Hodge star operator to obtain $\mathcal{H}^{m}_{2}(\tilde{M},\tilde{g})=\{0\}$. Summarizing $$\mathcal{H}^{k}_{2}(\tilde{M},\tilde{g})=\{0\}$$ for each $k=0,...,m$. Applying now  Atiyah's $L^2$-index theorem again, we have
\begin{equation}
\label{L2}
\chi(M)=\sum_{k=0}^m(-1)^k\dim_{\pi_1(M)}\left(\mathcal{H}^{k}_{2}(\tilde{M},\tilde{g})\right)=0
\end{equation}
and clearly this is in contradiction with the conclusion of the first step above, namely that $\chi(M)=2$. We can thus conclude that $\pi_1(M)$ is finite.\\

\noindent \textbf{Forth step} $M$ is simply connected.\\
 Once again let $\pi:\tilde{M}\rightarrow M$ be the universal covering of $M$. Then $(\tilde{M},\tilde{g})$ is a compact and oriented Riemannian manifold with $\mathcal{H}^{k}_2(\tilde{M},\tilde{g})=\{0\}$ for each $k=1,...,m-1$ and thus we have $\chi(\tilde{M})=2$. If we now denote with $\ell$ the cardinality of $\pi_1(M)$ we obtain by means of the Chern-Gauss-Bonnet theorem:
$$2=\chi(\tilde{M})=\int_{\tilde{M}}\mathrm{e}(\tilde{M})=\ell\int_{M}\mathrm{e}(M)=2\ell.$$
We can thus conclude that $\ell=1$ and so $M$ is simply connected, as required.
This completes the first part of the proof.\\

\noindent\textbf{Second part}.\\
Since $M$ is not orientable there exists a connected, oriented double covering $\pi:\hat{M}\rightarrow M$. Once again let us consider the operator 
\begin{equation}
\label{fvfvx}
\frac{m}{m-1}\hat{\Delta}+\hat{w}_{k,1}:C^{\infty}(\hat{M})\rightarrow C^{\infty}(\hat{M})
\end{equation}
with $\hat{w}_{k,1}:=w_{k,1}\circ \pi$, $\hat{\Delta}:C^{\infty}_c(\hat{M})\rightarrow C_c^{\infty}(\hat{M})$ the Laplace-Beltrami operator of $(\hat{M},\hat{g})$ and $\hat{g}=\pi^*g$. Arguing as in the proof of Th. \ref{finiteness} we know that \eqref{fvfvx} is  essentially self-adjoint in $L^2(\hat{M},\hat{g})$ and   its unique $L^2$-closed extension, denoted here with
\begin{equation}
\label{liftza}
\frac{m+1}{m}\hat{\Delta}+\hat{w}_{k,1}:L^2(\hat{M},\hat{g})\rightarrow L^2(\hat{M},\hat{g})
\end{equation}
has entirely positive spectrum. We are therefore in the position to apply the first part of this theorem to  conclude that $(\hat{M},\hat{g})$ is a simply connected real homology sphere. This in turn implies that $\pi:\hat{M}\rightarrow M$ is the universal covering of $M$ and so we can conclude that $\pi_1(M)\cong \mathbb{Z}_2$. Let now $\eta \in \Omega^k(M)$, $k=\{1,...,m-1\}$ be a closed form. Then there exists $\beta\in \Omega^{k-1}(\hat{M})$ such that $d_{k-1}\beta=\pi^*\eta$. Let $\gamma:\hat{M}\rightarrow \hat{M}$ be the non-trivial lift of the identity $\mathrm{Id}:M\rightarrow M$. Then $d_{k-1}(\gamma^*\beta)=\pi^*\eta$ and thus $\frac{1}{2}(d_{k-1}(\beta+\gamma^*\beta))=\pi^*\eta$. Since $\beta+\gamma^*\beta$ is invariant with respect to the action of the Deck transformations group of $\hat{M}$, we know that 
there exists $\alpha\in \Omega^{k-1}(M)$ such that $\frac{1}{2}(\beta+\gamma^*\beta)=\pi^*\alpha$. We can thus conclude that $\eta=d_{k-1}\alpha$, as required.\\

\noindent\textbf{Third part}. By repeating the same argument used in the first step of the first part we can conclude that $H^k(M,\mathbb{R})=\{0\}$ for $k=1,...,m-1$. On the other hand, given that $m$ is odd, we know that $\chi(M)=0$. This implies that $H^m(M,\mathbb{R})\cong \mathbb{R}$ and so we can conclude that $M$ is a real homology sphere. Finally $M$ is orientable as $H^m(M,\mathbb{R})\neq \{0\}$.
\end{proof}

\begin{rem}
Note that the finiteness of $\pi_1(M)$ in  Th. \ref{simplyr} can also be obtained by applying \cite[Th. 1]{AntoXu}.
\end{rem}

We continue now by collecting some geometric applications: 

\begin{cor} 
\label{ory}
In the setting of Th. \ref{simplyr} let us consider an arbitrarily fixed smooth map $f:M\rightarrow M$. Assume additionally that either $m$ is even, $M$ orientable and the degree of $f$ is different from $-1$ or that $M$ is not orientable.  Then $$\{x\in M : f(x)=x\}\neq \emptyset.$$ 
\end{cor}

\begin{proof}
Let $$L(f)=\sum_{k=0}^m(-1)^k\mathrm{Tr}\left(f^*: H^{k}_{dR}(M)\rightarrow H^{k}_{dR}(M)\right).$$
According to the Lefschetz fixed point theorem if $L(f)\neq 0$ then $f$ must have some fixed point. In our setting the above formula boils down to $$L(f)=1$$ if $M$ is not orientable and to $$L(f)=1+\mathrm{deg}(f)\neq 0$$ in the other case. We thus  can conclude that  $\{x\in M : f(x)=x\}\neq \emptyset$, as desired
\end{proof}

\begin{cor}
Let $M$ and $N$ be  compact,  manifolds of dimension $m>0$ and $n>0$. Then $M\times N$ carries no Riemannian metric $g$ such that $$\frac{m+n}{m+n-1}\Delta+w_{k,1}:L^2(M\times N,g)\rightarrow L^2(M\times N,g)$$ has entirely positive spectrum for each $k=1,...,m-1$.
\end{cor}

\begin{proof}
It is enough to prove the case where both $M$ and $N$ are not orientable. So let $\hat{\pi}_M:\hat{M}\rightarrow M$ and $\hat{\pi}_N:\hat{N}\rightarrow N$ be the corresponding oriented double coverings. By the fact that $\hat{M}\times \hat{N}$ is not a real homology sphere, namely $H^m(\hat{M}\times \hat{N},\mathbb{R})\neq \{0\}$ and $H^n(\hat{M}\times \hat{N},\mathbb{R})\neq \{0\}$, we can conclude that $\hat{M}\times \hat{N}$ carries no Riemannian metric $\hat{g}$ such that $$\frac{m+n}{m+n-1}\hat{\Delta}+w_{k,1}:L^2(\hat{M}\times \hat{N},\hat{g})\rightarrow L^2(\hat{M}\times \hat{N},\hat{g})$$ has entirely positive spectrum for each $k=1,...,m-1$. Finally, using the pushing-down technique, we can conclude that also $M\times N$ carries no Riemannian metric $g$ such that $\frac{m+n}{m+n-1}\Delta+w_{k,1}:L^2(M\times N,g)\rightarrow L^2(M\times N,g)$ has entirely positive spectrum for each $k=1,...,m-1$.
\end{proof}

We have now the Riemannian version of Th. \ref{finiteness}.

\begin{teo}
\label{finitenessr}
Let $(M,g)$ be a compact Riemannian manifold of  dimension $m>1$. Assume that $\chi(M)\neq 0$ and there exists  $\gamma\in [0,\frac{m}{m-1})$ such that the operator $$\gamma\Delta+w_{k,1}:L^2(M,g)\rightarrow L^2(M,g)$$ has entirely non-negative spectrum for each $k=1,...,m-1$. Then $M$ has finite fundamental group.
\end{teo}

\begin{proof}
We will be brief as the proof follows by adapting to this setting the proof of Th. \ref{finiteness}. Let us first deal with the case when $M$ is oriented. By contradiction let us assume that $\pi_1(M)$ is infinite. Let $\pi:\tilde{M}\rightarrow M$ be the universal covering of $M$. As in the previous proof let us consider the operator $$\gamma\tilde{\Delta}+\tilde{w}_{k,1}:C^{\infty}_c(\tilde{M})\rightarrow C_c^{\infty}(\tilde{M})$$ and, with a little abuse of notation, let 
\begin{equation}
\label{lift3}
\gamma\tilde{\Delta}+\tilde{w}_{k,1}:L^2(\tilde{M},\tilde{g})\rightarrow L^2(\tilde{M},\tilde{g})
\end{equation}
 its unique $L^2$-closed extension. Then by the pushing-down technique we know that \eqref{lift3} has still non-negative spectrum.  Let now $\psi\in \mathcal{H}^{k}_{2}(\tilde{M},\tilde{g})$, $k\in \{1,...,m-1\}$. Since $\psi\in W^{1,2}(\tilde{M}, \Lambda^{p,0}(\tilde{M}))$ and $|\psi|_{\tilde{g}}\in W^{1,2}(\tilde{M},\tilde{h})$ we can use the Weitzenb\"ock formula and the refined Kato inequality to  obtain
$$
\begin{aligned}
0&=\langle \tilde{\Delta}_k\psi,\psi\rangle_{L^2\Omega^{k}(\tilde{M},\tilde{g})}\\
&\geq \int_{\tilde{M}}\left(\frac{m}{m-1}|\tilde{d}_0|\psi|_{\tilde{g}}|^2_{\tilde{g}}+\tilde{w}_{k,1}|\psi|^2_{\tilde{g}}\right)\dvol_{\tilde{g}}\\
&\geq \int_{\tilde{M}}\left(\frac{m}{m-1}-\gamma\right)|\tilde{d}_0|\psi|_{\tilde{g}}|^2_{\tilde{g}}\dvol_{\tilde{g}}\geq 0.
\end{aligned}
$$ We can thus conclude that $|\psi|_{\tilde{g}}$ is a constant function on $\tilde{M}$ and since  $|\psi|_{\tilde{g}}\in L^2(\tilde{M},\tilde{g})$ and $(\tilde{M},\tilde{g})$ has infinite volume, we can deduce that $\psi\equiv 0$. This shows that $\mathcal{H}^{k}_{2}(\tilde{M},\tilde{g})=\{0\}$ for each $k=1,...,m-1$. The vanishing of $\mathcal{H}^0_2(\tilde{M},\tilde{g})$ follows by the fact that $(\tilde{M},\tilde{g})$ is complete with infinite volume while the vanishing of $\mathcal{H}^m_2(\tilde{M},\tilde{g})$ follows because the Hodge star operator $*$ induces an isomorphism between $\mathcal{H}^0_2(\tilde{M},\tilde{g})$ and $\mathcal{H}^m_2(\tilde{M},\tilde{g})$. Now by using Atiyah's $L^2$-index theorem as in the proof of Th. \ref{simplyr}, we deduce that $\chi(M)=0$ which contradicts the assumption $\chi(M)\neq 0$. We can thus conclude that $\pi_1(M)$ is finite, as required. If $M$ is not oriented let us consider a connected oriented double covering $\pi:\hat{M}\rightarrow M$. We know that $\chi(\hat{M})\neq 0$ and by the pushing-down technique we also know that  $$\gamma\hat{\Delta}+\hat{w}_{k,1}:L^2(M,g)\rightarrow L^2(M,g)$$ has entirely non-negative spectrum.  We can argue now as in the first part of this proof to conclude that $\pi_1(\hat{M})$ is finite and consequently we can conclude that also $\pi_1(M)$ is finite.
\end{proof}

\begin{cor}
In the setting of Th. \ref{finitenessr} every harmonic $k$-form on $M$ has constant length and consequently $$\dim(H^k_{dR}(M))\leq \frac{m!}{k!(m-k)!}$$ for each $k=1,...,m-1$.
\end{cor}
\begin{proof}
This follows immediately by Th. \ref{finitenessr}.
\end{proof}

\begin{rem}
Note that the hypothesis of Th. \ref{finitenessr}, namely that $\chi(M)\neq 0$, forces $m$ to be even.
\end{rem}

A sufficient condition to apply Th. \ref{simplyr} and Th. \ref{finitenessr} can be given in terms of the curvature operator. More precisely let $\mathcal{R}$ be the curvature operator of $(M,g)$. We recall that $\mathcal{R}$ is the endomorphism of $\Lambda^2TM$ such that for each $X,Y\in \mathfrak{X}(M)$ we have $\langle\mathcal{R}(X\wedge Y),X\wedge Y\rangle_g=g(R(X,Y)X,Y)$, with $R$ the curvature tensor of $(M,g)$. Let us denote with $$\sigma_1\leq...\leq \sigma_{n}$$ with $n=\frac{m(m-1)}{2}$ the continuous functions on $M$ such that $$\{\sigma_1(x)\leq...\leq \sigma_{n}(x)\}$$ are the eigenvalues of $\mathcal{R}_x:\Lambda^2T_xM\rightarrow \Lambda^2T_xM$ for each $x\in M$. We are interested in the Schr\"odinger operator $$\frac{4}{m^2-m}\Delta+\sigma_1:C^{\infty}(M)\rightarrow C^{\infty}(M)$$ and by adopting the same abuse of notation used before, we still denote with $$\frac{4}{m^2-m}\Delta+\sigma_1:L^2(M,g)\rightarrow L^2(M,g)$$ its unique closed (and thus self-adjoint) extension. We have the following results:

\begin{teo}
\label{curvop}
Let $(M,g)$ be a Riemannian manifold of dimension $m>1$. Assume that the operator 
\begin{equation}
\label{cuop}
\frac{4}{m^2-m}\Delta+\sigma_1:L^2(M,g)\rightarrow L^2(M,g)
\end{equation}
 has entirely positive spectrum and in the case $m\leq3$ that $M$ is additionally compact. We then have the following properties: 
\begin{enumerate}
 \item If $m$ is even and $M$ is orientable, then  $M$  is a compact, simply connected real homology sphere.
\item If $m$ is even and $M$ is not orientable, then it is a compact, real homology sphere with $\pi_1(M)\cong\mathbb{Z}_2$.
\item If $m$ is odd then $M$ is a real homology sphere. In particular $M$ is orientable.
\end{enumerate}
\end{teo}

\begin{proof}
Let us start by showing that if $m\geq 4$ then $M$ is compact. As in the previous section let us denote with $r_1:M\rightarrow \mathbb{R}$ the function that to each $x\in M$ assigns the smallest eigenvalue of $\mathrm{ric}_{g,x}:T_xM\rightarrow T_xM$, with $\mathrm{ric}_g$ the tangent bundle endomorphism such that $\mathrm{Ric}_g(\cdot,\cdot)=g(\mathrm{ric}_g\cdot,\cdot)$. Thanks to \cite[Cor. 2.6]{Gallot} we have 
$$
\begin{aligned}
\left\langle \frac{4}{m-1}\Delta\phi+r_{1}\phi,\phi\right\rangle_{L^2(M,g)}&\geq \left\langle \frac{4}{m-1}\Delta\phi+(m-1)\sigma_1\phi,\phi\right\rangle_{L^2(M,g)}\\
&=(m-1)\left\langle \frac{4}{(m-1)^2}\Delta\phi+\sigma_1\phi,\phi\right\rangle_{L^2(M,g)}
\end{aligned}
$$
Since \eqref{cuop} has entirely positive spectrum and  $\frac{1}{(m-1)^2}>\frac{1}{m^2-m}$ for each positive integer $m$, we can conclude that $$\frac{4}{(m-1)^{2}}\Delta+\sigma_1:L^2(M,g)\rightarrow L^2(M,g)$$ has entirely positive spectrum and thus that  also
$$\frac{4}{m-1}\Delta\phi+r_{1}:L^2(M,g)\rightarrow L^2(M,g)$$ has entirely positive spectrum. We can thus conclude that $M$ is compact thanks to \cite[Cor. 1]{AntoXu}. Let us now continue by denoting with $\mu$ the smallest eigenvalue of \eqref{cuop}. Clearly $\mu>0$ as we assumed that \eqref{cuop} has entirely positive spectrum. We want to show that 
\begin{equation}
\label{entpos}
\frac{m}{m-1}\Delta+w_{k,1}:L^2(M,g)\rightarrow L^2(M,g)
\end{equation}
 has entirely positive spectrum for each $k=1,...,m-1$. In order to accomplish this goal we prove that $$
\left\langle \frac{m}{m-1}\Delta\phi+w_{k,1}\phi,\phi\right\rangle_{L^2(M,g)}\geq k(m-k)\mu\|\phi\|^2_{L^2(M,g)}
$$
 for each $\phi\in C^{\infty}(M)$ and $k=1,...,m-1$.  Indeed, thanks to \cite[Cor. 2.6]{Gallot}, we have 
$$
\begin{aligned}
\left\langle \frac{m}{m-1}\Delta\phi+w_{k,1}\phi,\phi\right\rangle_{L^2(M,g)}&\geq \left\langle \frac{m}{m-1}\Delta\phi+k(m-k)\sigma_1\phi,\phi\right\rangle_{L^2(M,g)}\\
&=k(m-k)\left\langle \frac{m}{k(m-k)(m-1)}\Delta\phi+\sigma_1\phi,\phi\right\rangle_{L^2(M,g)}\\
& \geq k(m-k)\left\langle \frac{4}{m^2-m}\Delta\phi+\sigma_1\phi,\phi\right\rangle_{L^2(M,g)}\\
&\geq k(m-k)\mu\|\phi\|^2_{L^2(M,g)}.
\end{aligned}
$$ Now that we established that the operator \eqref{entpos} has entirely positive spectrum for each $k=1,...,m-1$, we can conclude by applying Th. \ref{simplyr}.
\end{proof}

\begin{teo}
\label{curvop2}
Let $(M,g)$ be a compact Riemannian manifold of  dimension $m>1$. Assume that $\chi(M)\neq 0$ and the operator $$\gamma\Delta+\sigma_1:L^2(M,g)\rightarrow L^2(M,g)$$ has entirely non-negative spectrum for some $\gamma\in [0,\frac{4}{m^2-m})$. Then $M$ has finite fundamental group.
\end{teo}

\begin{proof}
Since $\gamma\Delta+\sigma_1:L^2(M,g)\rightarrow L^2(M,g)$ has entirely non-negative spectrum for some $\gamma\in [0,\frac{4}{m^2-m})$ we also know that 
$$\gamma\tilde{\Delta}+\tilde{\sigma}_1:L^2(\tilde{M},\tilde{g})\rightarrow L^2(\tilde{M},\tilde{g})$$ has entirely non-negative spectrum for the same $\gamma\in [0,\frac{4}{m^2-m})$.
Let now $\psi\in \mathcal{H}_2^k(\tilde{M},\tilde{g})$ with $1\leq k\leq m-1$. We have $$
\begin{aligned}
0&=\langle \tilde{\Delta}_k\psi,\psi\rangle_{L^2\Omega^{k}(\tilde{M},\tilde{g})}\\
&\geq \int_{M}\left(\frac{m}{m-1}|\tilde{d}_0|\psi|_{\tilde{g}}|^2_{\tilde{g}}+\tilde{w}_{k,1}|\psi|^2_{\tilde{g}}\right)\dvol_{\tilde{g}}\\
&\geq k(m-k)\int_{M}\left(\frac{4}{m^2-m}|\tilde{d}_0|\psi|_{\tilde{g}}|^2_{\tilde{g}}+\tilde{\sigma}_1|\psi|^2_{\tilde{g}}\right)\dvol_{\tilde{g}}\\
&\geq k(m-k)\int_{M}\left(\frac{4}{m^2-m}-\gamma\right)|\tilde{d}_0|\psi|_{\tilde{g}}|^2_{\tilde{g}}\dvol_{\tilde{g}}\\
&\geq 0
\end{aligned}
$$
and so we can conclude that $|\psi|_{\tilde{g}}$ is a constant function on $\tilde{M}$. We can now repeat word by word the proof of Th. \ref{finitenessr} to achieve the desired conclusion.
\end{proof}

Finally, we conclude with the following  remarks.\\

\noindent Looking at Th. \ref{simplyr} the following question arises naturally: when does the operator 
\begin{equation}
\label{con+}
\frac{m}{m-1}\Delta+w_{k,1}:L^2(M,g)\rightarrow L^2(M,g)
\end{equation}
 has entirely positive spectrum? Below we collect some sufficient conditions  for the spectral positivity of \eqref{con+} that at the same time  allow $W_k$ to be entirely negative in some regions of $M$. Namely:\\

\noindent 1) $(M,g)$ has wells of negative $W_k$-curvature of suitable depth and volume, see \cite[Prop. 12]{ElRos};\\

\noindent 2) On $(M,g)$ the following inequalities hold true: 
$$\int_Mw_{k,1}\dvol_g>0\quad \mathrm{and}\quad P\leq \frac{m}{m-1}\frac{\int_M w_{k,1}\dvol_g}{\|w_{k,1}\|_{L^{\infty}(M)}(4\vol_g(M)\|w_{k,1}\|_{L^{\infty}(M)}+\int_M w_{k,1}\dvol_g)}$$ with $P$ denoting the Poincar\'e constant of $(M,g)$. See \cite[Cor. 1.2]{DabLock}.\\

\noindent 3) There exists a constant $w_0>0$ such that the function $\left(\frac{m-1}{m}w_{k,1}-w_0\right)_{-}$, that is the negative part of $\frac{m-1}{m}w_{k,1}-w_0$, satisfies a Kato-type condition: $$c_{\mathrm{Kato}}\left(\left(\frac{m-1}{m}w_{k,1}-w_0\right)_{-},w_0\right):=\left\|(\Delta+w_0)^{-1}\left(\frac{m-1}{m}w_{k,1}-w_0\right)_{-}\right\|_{L^{\infty}(M)}< 1.$$ Then, according to \cite[Cor. 2.3]{RoSt}, the operator $\Delta+w_0-\left(\frac{m-1}{m}w_{k,1}-w_0\right)_{-}$ has entirely positive spectrum and since $$\Delta+\frac{m-1}{m}w_{k,1}\geq \Delta+w_0-\left(\frac{m-1}{m}w_{k,1}-w_0\right)_{-}$$ in the sense of quadratic form, we can conclude that $\Delta+\frac{m-1}{m}w_{k,1}$, and thus also 
$\frac{m}{m-1}\Delta+w_{k,1}$, has entirely positive spectrum. We refer to \cite{CarRose} and \cite{RoSt} for more details.\\

\noindent 4) Let $\tau$ be the smallest eigenvalue of $\frac{m+1}{m}\Delta+w_{k,1,+}$, with $w_{k,1,+}$ the positive part of $w_{k,1}$. If $\tau>w_{k,1,-}(x)$ for each $x\in M$, with $w_{k,1,-}$ the negative part of $w_{k,1}$, then   $\frac{m}{m-1}\Delta+w_{k,1}$ has entirely positive spectrum. This follows easily by observing that given any $\phi\in W^{1,2}(M,g)$ with $\|\phi\|_{L^2(M,g)}=1$ we have 
$$
\begin{aligned}
\frac{m}{m-1}\|d_0\phi\|^2_{L^2\Omega^1(M,h)}+\langle w_{k,1}\phi,\phi\rangle_{L^2(M,g)}&=\frac{m}{m-1}\|d_0\phi\|^2_{L^2\Omega^1(M,g)}+\langle(w_{k,1,+}-w_{k,1,-})\phi,\phi\rangle_{L^2(M,g)}\\
& =\frac{m}{m-1}\|d_0\phi\|^2_{L^2\Omega^1(M,g)}+\langle w_{k,1,+}\phi,\phi\rangle_{L^2(M,g)}-\langle w_{k,1,-}\phi,\phi\rangle_{L^2(M,g)}\\&\geq\tau-\max_{M}w_{k,1,-}>0.
\end{aligned} $$

\noindent 5) On $(M,g)$ the following inequalities hold true: 
$$\int_Mw_{k,1}\dvol_g>0\quad \mathrm{and}\quad \frac{m+1}{m}\beta_1>-\min_M(w_{k,1})$$ with $\beta_1$ the first positive eigenvalue of the Laplace-Beltrami operator $\Delta:L^2(M,g)\rightarrow L^2(M,g)$. Indeed let $\phi\in W^{1,2}(M,g)$ with $\|\phi\|_{L^2(M,g)}=1$. By decomposing $\phi$ as $\phi=\phi_0+\phi_1$ with $\phi_0$ and $\phi_1$ orthogonal in $L^2(M,g)$ and $\phi_0$  constant we have
$$
\begin{aligned}
\frac{m}{m-1}\|d_0\phi\|^2_{L^2\Omega^1(M,g)}+\langle w_{k,1}\phi,\phi\rangle_{L^2(M,g)}&=\frac{m}{m-1}\|d_0\phi_1\|^2_{L^2\Omega^1(M,g)}+\int_M w_{k,1}\phi^2\dvol_g\\
& \geq \frac{m+1}{m}\beta_1+\min_{M}(w_{k,1})>0.
\end{aligned} $$

\newpage

\appendix
\section{Rational dimension, partial positivity and Bochner-type result, by Shiyu Zhang}

Proposition \ref{prop-vanishing} shows that the total $(m-k)$-positivity of the mean curvature of the tangent bundle implies that the rational dimension is at least $k$. Together with the main result of \cite{LZZ25}, this curvature positivity can be converted into the slope positivity
\[
\mu_\alpha(\Lambda^{m-k+1}T^{1,0}X)>0
\]
for some movable class $\alpha$; see, for example, Proposition \ref{prop-existence-hermitian}. The main purpose of this appendix is to prove the following equivalence, based on recent advances in \cite{ou25,CP25}.

\begin{teo}\label{prop-characterization-rationalconnectedness}
	For a compact K\"ahler manifold $X$ of complex dimension $m$, the condition $\rd(X)\geq k$ is equivalent to the existence of a movable class $\alpha$ such that
    \[
    \mu_\alpha(\Lambda^{m-k+1}T^{1,0}X)>0.
    \]
\end{teo}

In particular, we obtain the following characterization.

\begin{cor}\label{coro-characterization}
    For a compact K\"ahler manifold $X$ of complex dimension $m$, the condition $\rd(X)\geq k$ is equivalent to each of the following statements.
    \begin{itemize}
        \item The mean curvature of the tangent bundle is totally $(m-k)$-positive.
        \item For every $m-k+1\leq p\leq m$ and every pseudo-effective line bundle $\mL$,
        \[
        H^0(X,\Omega^{p,0}X\otimes \mL^{-1})=0.
        \]
    \end{itemize}
\end{cor}

Theorem \ref{prop-characterization-rationalconnectedness} was proved by Campana \cite{Campana16} for general orbifold pairs in the case $k=m$. The strategy is a further modification of \cite[Section 5]{Campana16}. One motivation for proving the above results from a differential-geometric viewpoint is to obtain Bochner-type results in terms of partial curvature positivity. It is well known that $k$-semi-positivity of the Ricci curvature implies that all holomorphic forms of degree at least $k$ are parallel. In \cite{Ni,NZ22}, it was shown that positivity of the $k$-scalar curvature implies the vanishing of all holomorphic forms of degree at least $k$; this has been extended to the semipositive case when $k=2$; cf. \cite{Ni25,Tang26}. Here, we obtain a complete result for arbitrary $k$.

\begin{teo}\label{teo-k-scalar}
    Let $(X,g)$ be a compact K\"ahler manifold of dimension $m$ with semi-positive $k$-scalar curvature, where $1\leq k\leq m$. Then for every $p\geq k$:
    \begin{itemize}
        \item[(1)] All holomorphic $p$-forms are parallell;
        \item[(2)] If additionally, $k$-scalar curvature of $g$ is positive at some point, then $h^{p,0}=0$.
    \end{itemize}
\end{teo}

\begin{proof}[Proof of Theorem \ref{teo-k-scalar}]
    Building on structure theorems established in \cite[Theorem 1.10]{ZZ25}, either $\rd(X)\geq m-k$, or the universal cover $(X_{\mathrm{univ}},g_{\mathrm{univ}})$ of $(X,g)$ is the product of a rationally connected manifold $(F,g_F)$ with semi-positive holomorphic sectional curvature and the universal cover $Y_{\mathrm{univ}}$ of a compact K\"ahler manifold $(Y,g_Y)$ with semi-positive Ricci curvature satisfies the following diagram:
    \[
    		\begin{tikzcd}
    			F_{\mathrm{univ}} & X_{\mathrm{univ}}=Y_{\mathrm{univ}}\times F_{\mathrm{univ}} \arrow[l,"\mathrm{pr}_2"'] \arrow[r] \arrow[rd,"\mathrm{pr}_1"'] \arrow[rr,bend left=10,"\mu"] & Y_{\mathrm{univ}}\times F \arrow[r] \arrow[d] & (Y_{\mathrm{univ}}\times F)/\pi_1(Y)\cong X \arrow[d,"f"] \\
    			& & Y_{\mathrm{univ}} \arrow[r,"\pi"] & Y.
    		\end{tikzcd}
    		\]
    Let $\eta$ be a holomorphic $(p,0)$-form with $p\geq k$. In the first case, any holomorphic $p$-form with $p\geq k$ are identically zero by Corollary \ref{coro-characterization}. In the second case, the pull-back $\widetilde{\eta}$ of $\eta$ on $\widetilde{X}$ must be holomorphic forms on $\widetilde{Y}$ because any holomorphic forms on $F$ are identically zero. In particular, $\eta$ is the pull-back of some holomorphic form on $Y$, and so is parallel because $\mathrm{Ric}(g_Y)\geq 0$. The proof is complete.
\end{proof}

It would be interesting to know whether the above Bochner-type result admits a proof avoiding MRC fibrations. We expect that such a proof should come from a differential-geometric approach.


\subsection{Basic definitions}
\subsubsection{Rational dimension}
Let $X$ be a compact K\"ahler manifold. We say that
\begin{enumerate}
	\item $X$ is uniruled if for every general point $x\in X$, there exists a rational curve containing $x$;
	\item $X$ is rationally connected if any two points can be joined by a (chain of) rational curve.
\end{enumerate}
A rationally connected manifold $X$ must be projective and shares important properties with Fano manifolds for example: Simply connectedness; admitting no nontrivial holomorphic forms. In this sense, rationally connected manifolds can be regarded as ``Fano-like." 

\vspace{0.1cm}

A maximally rationally connected (for short MRC) fibration $f:X\dashrightarrow Y$ of $X$ is an almost holomorphic map (that is meromorphic map whose general fibres are compact) from $X$ to a compact K\"ahler manifold $Y$ such that
\begin{itemize}
	\item General fibres are rationally connected;
	\item There is no horizontal rational curve passing through a general point in $X$.
\end{itemize}
The existence of a MRC fibration $f:X\dashrightarrow Z$ follows from the existence of a quotient map for covering families; see for instance \cite[Theorem 2.6]{Cam04}, where the quotient is defined via the equivalence relation that two points are equivalent if they can be joined by a rational curve.

\begin{defi}[Rational dimension]
	The rational dimension $\rd(X)$ of $X$ is defined by 
	$$\rd(X):=\dim X-\dim Y.$$
	In particular, $\rd(X)\geq1$ when $X$ is uniruled and $\rd(X)=n$ when $X$ is rationally connected.  	
\end{defi}

\subsubsection{Movable cone}
 The movable cone can be defined in the non-projective context (see e.g. \cite{ou25,CP25}).
\begin{defi}[The movable cone, cf. {\cite[Definition 4.2]{ou25}}]
	Let $X$ be a compact complex manifold of dimension $m$. The movable cone $\mov(X)\subset H_A^{m-1,m-1}(X,\R)$ is the closed cone generated by $[\w^{m-1}]$ for all Gauduchon metric $\w$ on $X$. A class $\alpha$ in $\mov(X)$ is called movable, and big if lies in the interior $\mov^0(X)$ of $\mov(X)$. $\mov^0(X)$ coincides with the Gauduchon cone.
\end{defi}
For any torsion-free sheaf $\mE$, the slope $\mu_\alpha(\mE)$ is defined by
$\mu_\alpha(\mE)=\frac{c_1(\mE)\cdot\alpha}{\rank \mE}.$ We also define
$$\mu_{\alpha,\max}(\mE):=\sup\{\mu_\alpha(\mF):0\subsetneq\mF\subset\mE\}$$
and
$$\mu_{\alpha,\min}(\mE):=\inf\{\mu_\alpha(\mQ):0\neq \mQ\text{ is a torsion-free quotient of } \mE\}.$$ 
As discussed in \cite[Lemma 4.3]{ou25}, $\mu_{\alpha,\max}(\mE)$ and $\mu_{\alpha,\min}(\mE)$ can be achieved for some saturated subsheaf $\mF$ and torsion-free quotient $\mQ$ and in particular are finite real numbers.  Then the Harder-Narasimhan filtration with respect to $\alpha$ exists. Suppose that $\mE$ is a torsion-free sheaf of $\rank$ $r$ on a compact complex manifold $X$. For any movable class $\alpha$, $\mE$ admits a unique filtration of saturated subsheaves
$$0=\mE_{0}\subsetneq \mE_{1}\subsetneq\cdots\subsetneq\mE_{l}=\mE$$
such that each quotient sheaf $\mQ_{k}=\mE_{k}/\mE_{k-1}$ is torsion-free and $\mu_{\alpha}(\mQ_k)>\mu_{\alpha}(\mQ_{k+1})$, and $\mQ_k$ is $\alpha$-semistable. Such a filtration is called the Harder-Narasimhan (for short HN) filtration of $\mE$ with respect to $\alpha$.
\begin{defi}
	We define the HN type of $\mE$ with respect to $\alpha$ by
	$$\vec{\mu}_\alpha(\mE)=(\mu_{\alpha,1}(\mE),\cdots,\mu_{\alpha,r}(\mE))$$
	where $\mu_{\alpha,i}(\mE):=\mu_\alpha(\mQ_k)$ for $\rank\mQ_{k-1}+1\leq i \leq \rank \mQ_k$.	In particular, $\mu_{\alpha,1}(\mE)=\mu_{\alpha,\max}(\mE)$, $\mu_{\alpha,r}(\mE)=\mu_{\alpha,\min}(\mE)$ and $\mu_{\alpha,1}(\mE)\geq \cdots\geq \mu_{\alpha,r}(\mE)$.
\end{defi}

We have the following statement (cf. \cite[Appendix]{CP25}):

\begin{prop}\label{prop-slope-computation}
	Given two torsion-free sheaves $\mE,\mF$ on a compact complex manifold $X$ and $\alpha$ be a big movable class, then for any $1\leq k \leq \rank\mE$,
	$$\mu_{\alpha,\min}((\Lambda^{k}\mE)^{\vee\vee})=\mu_{\alpha,r-k+1}(\mE)+\cdots+\mu_{\alpha,r}(\mE),$$
	and
	$$\mu_{\alpha,\min}((\mE\otimes\mF)^{\vee\vee})=\mu_{\alpha,\min}(\mE)+\mu_{\alpha,\min}(\mF).$$
\end{prop}

\subsubsection{Mean curvature}
Let $(X,g)$ be a Hermitian manifold of complex dimension $m$ and $(E,H)$ be a Hermitian vector bundle of $\rank$ $r$ on $X$, the mean curvature $\tr_gR^H\in \End(E)$ of $(E,H)$ with respect to $g$ is defined by
$$(\tr_gR^H)_{k}^l:=g^{i\bar{j}}R_{i\bar{j}k}^l,$$
where $R^H$ is the Chern curvature tensor of $h$.
\begin{defi}[Mean curvature]
	Let $E$ be a holomorphic vector bundle of $\rank r$ on $X$. We say that
		 $E$ is mean curvature  $(p-1)$-positive if there exists a Hermitian metric $H$ on $E$ and a Guduchon metric $g$ on $X$ such that
		$$\frak{r}_{p,g}(E,H):=\lambda_{r-p+1}+\cdots+\lambda_{r}$$ of $\tr_gR^H$ is positive over $X$, where $\lambda_{r-p+1}\geq \cdots \geq\lambda_{r}$ denote the smallest $p$ eigenvalues of $\tr_gR^H$; $E$ is total mean curvature  $(p-1)$-positive if $$\int_X\frak{r}_{p,g}(E,H)\cdot\frac{\w_g^m}{m!}>0.$$
\end{defi}

The implication $(2)$ is a direct consequence of the following statement.

\begin{prop}\label{prop-existence-hermitian}
   Let $E$ be a holomorphic vector bundle of rank $r$ on a compact complex manifold $X$ of complex dimension $m$ and $\alpha$ be a movable class. If $\mu_{\alpha,\min}(\Lambda^{p}E)>0$ for some $0<p\leq r$, then $E$ is mean curvature $(p-1)$-positive. 
\end{prop}
\begin{proof}
	Let $\beta$ be some fixed big movable class.
	For any quotient torsion-free sheaf $\mQ$ of $\Lambda^p(E)$,
	\begin{equation}
		\mu_{\alpha+\epsilon \beta}(\mQ)=\mu_{\alpha}(\mQ)+\epsilon\mu_\beta(\mQ)\geq \mu_{\alpha,\min}(\mQ)+\epsilon \mu_{\beta,\min}(\mQ)>0,
	\end{equation}
	so $\mu_{\alpha+\epsilon\beta,\min}(\Lambda^{p}E)>0$ for sufficiently small $\epsilon$ and we may assume that $\alpha$ is big.
	
	Let $g$ be a Gauduchon metric such that $\alpha=[\w_g^{m-1}]$ and $H(t)$ be the solution of the Hermitian-Yang-Mills flow on $E$ with respect to $g$, then $\Lambda^{p}H(t)$ is the solution of the Hermitian-Yang-Mills flow on $\Lambda^{p}E$. It follows from \cite[Theorem 1.1]{CLZZ26} and \cite[Theorem 1.3]{LZZ25} that the smallest $p$ eigenvalue of $\tr_gR^{\Lambda^{p}H(t)}$ increases to $\mu_{\alpha,\min}(\Lambda^{p} E)$. Thus, there exists some Hermitian metric $H$ on $E$ such that $\tr_gR^{\Lambda^{p}H}$ is positive definite and so $\tr_g R^H$ is $(p-1)$-positive.
\end{proof}

\subsection{The slope positivity}
\subsubsection{Relative movable cone}Let us first introduce the definition of relative Movable class. Let $f:X\rightarrow Z$ be a fibration (i.e. a surjective holomorphic map with connected fibers) between two compact complex manifolds and denote $m:=\dim X$ and $n:=\dim Z$. There is a natural surjective direct image map: $f_*:H_A^{m-1,m-1}(X,\R)\rightarrow H_A^{n-1,n-1}(Z,\R)$ dual to the inverse image map $f^*:H_{BC}^{1,1}(Z,\R)\rightarrow H_{BC}^{1,1}(X,\R)$. The kernel $H_A^{m-1,m-1}(X/Z,\R)$ of $f_*$ is the orthogonal of the image of $f^*$.

\begin{defi}
    The relative movable cone $\mov(X/Z)$ is defined by $\mov(X)\cap H_A^{m-1,m-1}(X/Z,\R)$ and denote its interior by $\mov^0(X/Z)$. 
\end{defi}

The main content is establishing the K\"ahler version of \cite[Proposition 3.6]{Campana16}.

\begin{prop}\label{prop-key}
    Let $f:X\rightarrow Z$ be a fibration between two compact K\"ahler manifolds and denote $m:=\dim X$ and $d:=\dim X-\dim Z$. Let $\w_X,\w_Z$ be K\"ahler classes on $X$ and $Z$, respectively. Let $F:=\cup_{s\in S}F_s$ be an effective reduced divisor on $X$. Let $K\subset S$ such that $f(F_k)=E_k$ is a divisor of $Z$ and $f^{-1}(E_k)\cap F\subsetneq f^{-1}(E_k)$. Suppose that $F$ is partially supported on the fibers of $f$, i.e., for any $s\in S\setminus K$, $E_s$ is $f$-exceptional, then the following statements hold.
    \begin{enumerate}
        \item[(1)]   $M_{kt}:=[F_k]\cdot[F_t]\cdot [\w_X]^{d-1}\cdot[f^*\w_Z]^{m-d-1}$ ($k,t\in K$) is a strictly negative matrix and moreover $M_{ts}\geq 0$ if $t\neq s$.
        \item[(2)] Let $\beta\in\mov(Z)$. There exists $\beta'\in\mov(X)$ such that: $f_*(\beta')=\beta$, and moreover: $\beta'\cdot [F_s]=0, \forall s\in S$.
    \end{enumerate}
\end{prop}

\begin{proof}
    Since
\[
T_{kt}:=f_*\big([F_k]\cdot [F_t]\cdot [\omega_X]^{d-1}\big)
\]
is a closed \((1,1)\)-current supported in \(E_k\cap E_t\), it vanishes if
\(E_k\ne E_t\). Hence \(M_{kt}=0\) unless \(E_k=E_t\). Therefore, to show the negativity of $(M_{kt})_{k,t\in K}$, it suffices to consider the case where $f(F_k)=E_k=E$ is the same $E$, for all $k\in K$. There exist the index $J$ and some $a_j>0$ ($j\in J$) such that $f^*E=\sum\limits_{j\in J} a_j F_j$ and $K\subset J$. Set $$M_{ij}:=[F_i]\cdot [F_j]\cdot[\w_X]^{d-1}\cdot [f^*\w_Z]^{m-d-1}.$$   Note that
    \begin{align*}
    \sum\limits_{j\in J}a_j M_{ij}&=[F_i]\cdot [f^*E]\cdot[\w_X]^{d-1}\cdot[f^*\w_Z]^{m-d-1}\\
    &=\int_{F_{i,\text{reg}}}(\w_X|_{F_{i,reg}})^{d-1}\wedge (f|_{F_{i,\text{reg}}})^*(c_1(\mathcal{O}(E),h)\wedge \w_Z^{m-d-1})\\
    &=0
    \end{align*}
    for some Hermitian metric $h$ of $\mathcal{O}(E)$ because $(f|_{F_{k,\text{reg}}})^*(c_1(\mathcal{O}(E),h)\w_Z^{m-d-1})=0$. Furthermore, for any $i\neq j$, $[F_i]\cdot[F_j]=m_p[C_{ijp}]$ for some $m_p\in\mathbb{N}^*$, where $C_{ijp}$ is the irrducible component of $F_i\cap F_j$, thus $M_{ij}\geq 0$ for any $i\neq j$ and $M_{ij=0}$ if and only if $\dim f(F_i\cap F_j)\leq m-d-2$. 
    Then for any $\{x_j\}\in \R^{|J|}$, we have
    $$-\sum\limits_{i,j\in J} x_i x_j M_{ij}=\sum\limits_{i\in J}a_i (\frac{x_i}{a_i})^2(\sum\limits_{j\in J}a_i M_{ij})-\sum\limits_{i,j\in J} x_i x_j M_{ij}=\frac{1}{2}\sum\limits_{i,j\in J, i\neq j}M_{ij}a_ia_j (\frac{x_i}{a_i}-\frac{x_j}{a_j})^2\geq 0$$
    and so $(M_{ij})_{i,j\in J}$ is a semi-negative matrix. We claim that the graph \(\Gamma_J\), whose vertices are \(J\) and whose edges are given by
\(M_{ij}>0\), is connected, which shall complete the proof of (1). It follows that
\[
-\sum_{i,j\in J}x_ix_jM_{ij}
=
\frac12\sum_{i\ne j}M_{ij}a_ia_j
\left(\frac{x_i}{a_i}-\frac{x_j}{a_j}\right)^2\ge 0.
\]
Since \(\Gamma_J\) is connected, equality holds only when \(\frac{x_i}{a_i}\) is constant.
Thus the kernel of \((M_{ij})_{i,j\in J}\) is generated by \((a_i)_{i\in J}\).
Since \(K\subsetneq J\), the principal submatrix \((M_{kt})_{k,t\in K}\) is negative definite. Now let us focus on showing the claim. We argue by contradiction. Otherwise, there exist nonempty sets $J_1$ and $J_2$ such that $J=J_1\sqcup J_2$ and $M_{j_1j_2}=0$ for any $j_1\in J_1,j_2\in J_2$. Let $B=\bigcup\limits_{j_1\in J_1,j_2\in J_2}f(F_i\cap F_i)$. $\dim B\leq m-d-2$ implies $B\subsetneq E$. Fix $z\in E\setminus B$, $\left(F_{j_1}\cap X_z\right)\bigcap\left(F_{j_2}\cap X_z\right)= \emptyset$ for any $j_1\in J_1,j_2\in J_2$, where $X_z=f^{-1}(z)$. Then
    $$X_z= \left(\bigcup\limits_{j_1\in J_1}F_{j_1}\cap X_z\right)\bigsqcup \left(\bigcup_{j_2\in J_2}F_{j_2}\cap X_z\right),$$
    which contradicts the assumption that $X_z$ is connected. The proof of (1) is complete.

    Now let us show (2). We may assume that $f_*([\w_X]^d)=1$. Let $$\beta'=[\w_X]^{d}\cdot f^*\beta+\sum\limits_{k\in K}c_k[F_k]\cdot [\w_X]^{d-1}\cdot[f^*\w_Z]^{m-d-1},$$ where $c_k\in \R$ are to be determined. One obviously has: $f_*\beta'=\beta$ and $\beta'\cdot [F_s]=0$ for any $s\in S\setminus K$. Building on the statement (1), $(M_{kt})_{k,t\in K}$ is invertible and every coefficient of its inverse matrix must be positive. Then, we obtain the existence and non-negativity of the solutions $c_k,k\in K$, of the equations:
    $$-(\sum\limits c_k[F_k]\cdot[\w_X]^{d-1}\cdot[f^*\w_Z]^{m-d-1})\cdot[F_t]=[\w_X]^d\cdot[f^*\beta]\cdot[F_t], \forall t\in K.$$
   One can see that $\beta'\cdot[F_k]=0$ for all $k\in K$ by the construction, and $F_s, s\in S\setminus K$ is $f$-exceptional divisor implies $\beta'\cdot[F_s]=0$. It remains to check that $\beta'$ is movable. For any pseudo-effective class $\alpha$, there exists a unique divisorial Zariski decomposition $$\alpha=P(\alpha)+[N(\alpha)]$$ such that $P(\alpha)$ is modified nef and $N(\alpha)=\sum\limits_{\text{prime divisor }D}v(\alpha,D)D$ (cf. \cite{Boucksom04}). For any effective prime divisor \(D\), we have \(\beta'\cdot D\ge 0\).
Indeed, if \(D=F_s\) for some \(s\in S\), this follows from the construction above.
If \(D\) is different from all \(F_k\), then every correction term has
non-negative intersection with \(D\), while
\([\omega_X]^d\cdot f^*\beta\cdot D\ge0\).
Thus \(\beta'\cdot [N(\alpha)]\ge0\).   For any $\epsilon>0$, we can take a positive current $T$ in $P(\alpha)$ with analytic singularities in codimension at least $2$ such that $T\geq -\epsilon\w_X$ (c.f. \cite[Theorem 2.1]{Boucksom04}). Then one can see that 
    $$[F_k]\cdot[\w_X]^{d-1}\cdot[f^*\w_Z]^{m-d-1}\cdot P(\alpha)\geq -\epsilon\int_{F_{k,\text{reg}}} \w_X^{d}\wedge f^*\w_Z^{m-d-1}$$
    for every $\epsilon>0$. We conclude that $$[F_k]\cdot[\w_X]^{d-1}\cdot[f^*\w_Z]^{m-d-1}\cdot\alpha=[F_k]\cdot[\w_X]^{d-1}\cdot[f^*\w_Z]^{m-d-1}\cdot P(\alpha)+[F_k]\cdot[\w_X]^{d-1}\cdot[f^*\w_Z]^{m-d-1}\cdot [N(\alpha)]\geq 0.$$  Note that $[\w_X]^d\cdot[f^*\beta]$ is movable, one has that $[\w_X]^d\cdot[f^*\beta]\cdot\alpha\geq 0$ and thus $\beta'\cdot\alpha\geq 0$. The proof is complete.
\end{proof}

\subsubsection{Proof of Theorem \ref{prop-characterization-rationalconnectedness}}
Building on Corollary \ref{coro-rationaldimension} and Proposition \ref{prop-existence-hermitian}, it remains to show that $\rd(X)\geq k$ implies $\mu_\alpha(\Lambda^{m-k+1}T^{1,0}X)>0$ for some movable class $\alpha$.

	 We will argue by induction. Suppose that (1) holds whenever $\dim X<m$. Since $X$ is uniruled, $K_X$ is not pseudo-effective by \cite[Theorem 1.1]{ou25} and thus $\alpha_0\cdot K_X<0$ for some movable class $\alpha_0$. Let $\mG\subset T^{1,0}X$ be the $\alpha_0$-maximal destablising subsheaf. It defines a rational fibration (by \cite[Corollary 1.5 and Corollary 6.7]{CP25}) $f:X\dashrightarrow Y$ with $\mG=\ker(df)$. Moreover, this fibration is non-trivial: $\dim Y<m$. Put $d:=\dim X-\dim Y=\rank (\mG)$.
	
	We consider bimeromorphic maps $\pi_X:\widehat{X}\rightarrow X,$ $\pi_Y:\widehat{Y}\rightarrow Y$ and a holomorphic map $\widehat{f}:\widehat{X}\rightarrow \widehat{Y}$ such that $f\circ\pi_X=\pi_Y\circ\widehat{f}$ and the following statements are satisfied.
	\begin{itemize}
		\item The discriminant locus of $\widehat{f}$ is a divisor $E$ with snc support
		\item If a component of $\widehat{f}^{-1}(E)$ is exceptional, then it is also $\pi_X$-exceptional.
	\end{itemize}
	\[
	\begin{tikzcd}
		\widehat{X} \arrow[r,"\pi_X"] \arrow[d,"\widehat{f}"]	& X \arrow[d,"f",dashrightarrow]  \\
		\widehat{Y} \arrow[r,"\pi_Y"]	& Y
	\end{tikzcd}
	\]
	$\mG=\ker(df)$ induces a foliation $\widehat{\mG}$ on $\widehat{X}$. i.e., $\widetilde{\mG}=\pi_X^*\mG\cap T^{1,0}\widehat{X}$ is saturated in $T^{1,0}\widehat{X}$, and it is known that $\det(\widehat{\mG}^\vee)=K_{\widehat{X}/\widehat{Y}}-D(\widehat{f})$ nodule $\pi_X$-exceptional divisor (c.f. \cite[Section 2]{Druel17}). Then
	$$0>-\alpha \cdot \det(\widehat{\mG}^\vee)=\alpha\cdot(K_{\widehat{X}/\widehat{Y}}-D(\widehat{f})),$$
	where $\alpha=\pi_X^*\alpha_0$, so $K_{\widehat{X}/\widehat{Y}}-D(\widehat{f})$ is not pseudo-effective. Since $\widehat{f}$ satisfies the hypothesis of \cite[Theorem 3.1]{CP25} (see e.g. \cite[Page 16]{CP25}), $K_{\widehat{X}_y}+D_{y}:=(K_{\widehat{X}/\widehat{Y}}-D(\widehat{f}))|_{\widehat{X}_y}$ is not pseudo-effective for generic $y\in\widehat{Y}$. Then there exists some $\alpha'\in\mov^0(\widehat{X}/\widehat{Y})$ such that $\alpha'\cdot K_{X/Z}<0.$ Let $\mF'$ be the $\alpha'$-maximal destablizing subsheaf. Following the argument of \cite[Page 36]{Campana16}, $\mF'\subset \ker(d\widehat{f})$ is an algebraic foliation. Assume that $d_{\alpha_0}:=d$ is minimal among all $d_{\alpha_0}$. Then $\mF'=\ker(df)$ and $\mu_{\alpha',\min}(\mF')>0$ by the construction.
    
    Consider the following short exact sequence
    $$0\rightarrow \mF'\rightarrow T^{1,0}\widehat{X}\rightarrow \mQ\rightarrow 0,$$
    which induces a filtration $\{\mE_p\}$ of $\Lambda^{m-k+1}T^{1,0}\widehat{X}$ by $\mF_p:=\im(\Lambda^{m-k+1-p}\mF'\otimes\Lambda^{p}\mQ\rightarrow \Lambda^{m-k+1}T^{1,0}\widehat{X})$ such that
    $$(\mE_{p}/\mE_{p-1})^{\vee\vee}\cong (\Lambda^{m-k+1-p}\mF'\otimes\Lambda^{p}\mQ)^{\vee\vee}.$$ It follows from Proposition \ref{prop-slope-computation} that for any big movable class $\alpha$,
    $$\mu_{\alpha,\min}((\Lambda^{m-k+1-p}\mF'\otimes\Lambda^{p}\mQ)^{\vee\vee})\geq (m-k+1-p)\mu_{\alpha,\min}(\mF')+\mu_{\alpha,\min}((\Lambda^p\mQ)^{\vee\vee}).$$
    It suffices to find some movable class $\alpha\in\mov(X)$ such that \begin{equation}\label{equa-desiredproperty}
        (m-k+1-p)\mu_{\alpha,\min}(\mF')+p\mu_{\alpha,\min}(\mQ^{\vee\vee})>0,\ \forall 0\leq p\leq  m-d.
    \end{equation}
    
    Since $\mQ$ is a subsheaf of $\widehat{f}^*T^{1,0}\widehat{Z}$ and $\widehat{f}^*T^{1,0}\widehat{Z}/\mQ$ is supported in a divisor $F$ such that $\widehat{f}(F)\subsetneq W$. Then for generic $z\in Z$ and any quotient sheaf $\mH''$ of $\mQ$, there exists some quotient sheaf $\mH'$ of $\widehat{f}^*T^{1,0}\widehat{Z}$ such that $\mH''\subset\mH'$ and $\mH'|_{\widehat{f}^{-1}(z)}=\mH'_{\widehat{f}^{-1}(z)}$, which implies that $\alpha'\cdot\det(\mH')=\alpha'\cdot\det(\mH'')$ since $\alpha'\in \mov^0(\widehat{X}/\widehat{Z})$ and so $\mu_{\alpha',\min}(\mQ)\geq \mu_{\alpha',\min}(\widehat{f}^*T^{1,0})(\widehat{Z})$. For any quotient sheaf $\mH'$ of $(\widehat{f}^* T^{1,0}\widehat{Z})$, Lemma 5.3 in \cite{Campana16} implies that either $\det(\mH')\geq \widehat{f}^*\det (\mH)$ for some quotient sheaf $\mH$ of $T^{1,0}\widehat{Z}$, or $\det(\mH')$ is a non-zero effective divisor when redtricted to the generic fiber of $\widehat{f}$. Thus, in the first case, $\mu_{\alpha',\min}(\widehat{f}^*T^{1,0}Z)\geq 0$ and in the second case $\mu_{\alpha_t}(\mH')\geq \mu_{\alpha,\min}(\mH)>0$. Together, we conclude that $\mu_{\alpha',\min}(\mQ)\geq 0$. When $k=d$, $\alpha'$ satisfies \eqref{equa-desiredproperty} by combining $\mu_{\alpha',\min}(\mF')>0$ and Proposition \ref{prop-slope-computation}.

    Now let us consider the case $k<d$. Since $\rd(\widehat{Y})=\rd(\widehat{X})-d\geq k-d$ and $\dim \widehat{Y}<m$, we can apply induction: there exists some movable class $\beta\in \mov^0(\widehat{Y})$ such that $\mu_{\beta,\min}(\Lambda^{m-k+d+1}(T^{1,0}\widehat{Y}))>0$.
     Building on Proposition \ref{prop-key}, there exists some $\beta'\in\mov(X)$ such that $\widehat{f}_*(\beta')=\beta$, and moreover: $\beta'\cdot[F_s]=0$, forall $s\in S$. Define $\alpha_t:=\alpha'+t\beta'$ for $t>0$, then building on the similar discussion as above, for any quotient sheaf $\mH''$ of $\Lambda^p(\mQ^{\vee\vee})$, either, $\mu_{\alpha_t}(\mH'')\geq t\mu_{\beta,\min}(\Lambda^pT^{1,0}\widehat{Z})>0$ or $\mu_{\alpha_t}(\mH'')>0$. Hence, one has
     $$\mu_{\alpha_t,\min}((\Lambda^{m-k+1}\mQ)^{\vee\vee})>0,$$
     and for any $0\leq p< m-k+1$,
     \begin{align*}
         &(m-k+1-p)\mu_{\alpha_t,\min}(\mF')+\mu_{\alpha_t,\min}((\Lambda^p\mQ)^{\vee\vee})\\
         \geq& (m-k+1-p)\mu_{\alpha',\min}(\mF')+t(m-k+1-p)\mu_{\beta',\min}(\mF')+p\max\{t\mu_{\beta,\min}(\mQ),0\}\\
         >&0
     \end{align*}
     when $t>0$ is sufficiently small. The proof is complete.

\end{document}